\documentclass[reqno,12pt]{amsart}

\usepackage{amsmath, amssymb}
\usepackage[hypertex]{hyperref}
\usepackage{mathrsfs}
\usepackage{color}

\numberwithin{equation}{section}

\newcommand{\1}{\mathbf {1}}
\newcommand{\0}{\mathbf {0}}
\newcommand{\Z}{{\mathbb Z}}
\newcommand{\Q}{{\mathbb Q}}
\newcommand{\R}{{\mathbb R}}
\newcommand{\C}{{\mathbb C}}

\newcommand{\h}{{\mathfrak h}}
\newcommand{\wh}{{\hat{\mathfrak h}}}
\newcommand{\CA}{{\mathcal A}}
\newcommand{\CC}{{\mathcal C}}
\newcommand{\CD}{{\mathcal D}}
\newcommand{\CK}{{\mathcal K}}

\newcommand{\bsa}{\boldsymbol{a}}
\newcommand{\bu}{\boldsymbol{u}}
\newcommand{\bv}{\boldsymbol{v}}

\newcommand{\al}{\alpha}

\newcommand{\lm}{\lambda}

\newcommand{\vep}{\varepsilon}
\newcommand{\ep}{\epsilon}
\newcommand{\la}{\langle}
\newcommand{\ra}{\rangle}

\newcommand{\Dih}{\mathrm{Dih}}
\newcommand{\Sym}{\mathrm{Sym}}

\newcommand{\hn}{\hat{\nu}}

\newcommand{\olL}{\overline{L}}
\newcommand{\ola}{\overline{\alpha}}
\newcommand{\olb}{\overline{\beta}}
\newcommand{\olu}{\overline{u}}
\newcommand{\olv}{\overline{v}}
\newcommand{\olbu}{\overline{\boldsymbol{u}}}
\newcommand{\olbv}{\overline{\boldsymbol{v}}}

\DeclareMathOperator{\Ann}{Ann}
\DeclareMathOperator{\Aut}{Aut}

\DeclareMathOperator{\Hom}{Hom}
\DeclareMathOperator{\Irr}{Irr}
\DeclareMathOperator{\Ker}{Ker}
\DeclareMathOperator{\pr}{pr}

\DeclareMathOperator{\rank}{rank}

\DeclareMathOperator{\spn}{span}

\DeclareMathOperator{\Weyl}{Weyl}
\DeclareMathOperator{\wt}{wt}

\newcommand{\SC}[1]{\Irr(#1)_{\mathrm{sc}}}

\newtheorem{theorem}{Theorem}[section]
\newtheorem{proposition}[theorem]{Proposition}
\newtheorem{lemma}[theorem]{Lemma}
\newtheorem{corollary}[theorem]{Corollary}
\newtheorem{remark}[theorem]{Remark}

\begin{document}

\title[sigma involutions]
{Sigma involutions associated with parafermion vertex operator algebra 
$K(\mathfrak{sl}_2,k)$}

\author[C.H. Lam]{Ching Hung Lam}
\address{Institute of Mathematics, Academia Sinica, Taipei 10617, Taiwan}
\email{chlam@math.sinica.edu.tw}

\author[H. Yamada]{Hiromichi Yamada}
\address{Department of Mathematics, Hitotsubashi University, Kunitachi,
Tokyo 186-8601, Japan}
\email{yamada.h@r.hit-u.ac.jp}

\keywords{vertex operator algebra, fusion algebra, automorphism}

\subjclass[2010]{Primary 17B69; Secondary 17B67, 17B68}

\begin{abstract}
An irreducible module for the parafermion vertex operator algebra 
$K(\mathfrak{sl}_2,k)$ is said to be of $\sigma$-type if an automorphism 
of the fusion algebra of $K(\mathfrak{sl}_2,k)$ of order $k$ is trivial on it. 
For any integer $k \ge 3$, we show that there exists an automorphism of order $2$ 
of the subalgebra of the fusion algebra of $K(\mathfrak{sl}_2,k)^{\la \theta \ra}$ 
spanned by the irreducible direct summands of 
$\sigma$-type irreducible $K(\mathfrak{sl}_2,k)$-modules, 
where $\theta$ is an involution of $K(\mathfrak{sl}_2,k)$. 
We discuss some examples of such an automorphism as well. 
\end{abstract}

\maketitle

\section{Introduction}\label{sec:intro}

Symmetry in the fusion algebra of a vertex operator algebra is an important 
subject in the representation theory of vertex operator algebras. 
The subalgebra of the fusion algebra on which the symmetry is trivial may 
admit another symmetry. 
Our main concern is the symmetry in such a subalgebra which appears 
if we consider the subalgebra as a subalgebra of the fusion algebra of an orbifold of 
the original vertex operator algebra. 

The fusion algebra of the parafermion vertex operator algebra 
$K(\mathfrak{sl}_2,k)$ associated with $\mathfrak{sl}_2$ and an integer $k \ge 2$
has an automorphism of order $k$.  
An irreducible $K(\mathfrak{sl}_2,k)$-module is said to be of $\sigma$-type 
if the automorphism is trivial on it. 
If $k \ge 3$, then the automorphism group of $K(\mathfrak{sl}_2,k)$ 
is of order $2$, and there is a unique involution $\theta$. 
Any $\sigma$-type irreducible $K(\mathfrak{sl}_2,k)$-module is $\theta$-stable, 
and it is a direct sum of two irreducible modules for the fixed point subalgebra 
$K(\mathfrak{sl}_2,k)^{\la \theta \ra}$.

In this paper, we show that there exists an automorphism of order $2$ 
of the subalgebra of the fusion algebra of $K(\mathfrak{sl}_2,k)^{\la \theta \ra}$ 
spanned by the irreducible direct summands of $\sigma$-type irreducible 
$K(\mathfrak{sl}_2,k)$-modules. 
We call such an automorphism a $\sigma$-involution. 
We develop a general theory of $\sigma$-involutions. 
Interesting examples of $\sigma$-involutions appear as automorphisms of 
certain lattice vertex operator algebras.  
We discuss the relationship between those $\sigma$-involutions and 
isometries of the underlying lattices. 

The parafermion vertex operator algebra $K(\mathfrak{sl}_2,k)$ 
is by definition the commutant of the Heisenberg vertex operator algebra 
generated by the Cartan subalgebra of $\mathfrak{sl}_2$ in 
the simple affine vertex operator algebra associated with the affine 
Kac-Moody algebra $\widehat{\mathfrak{sl}}_2$ at an integral level $k \ge 2$, 
see Section \ref{subsec:paraf_VOA} for basic properties of $K(\mathfrak{sl}_2,k)$. 
It is a simple, self-dual, rational, and $C_2$-cofinite vertex operator algebra 
of CFT-type with central charge $2(k-1)/(k+2)$. 
The irreducible $K(\mathfrak{sl}_2,k)$-modules are denoted by $M^{i,j}$ 
with $0 \le i \le k$ and $0 \le j < k$, where $j$ is considered to be an integer modulo $k$. 
We have $M^{i,j} \cong M^{k-i, j-i}$, 
and $M^{i,j}$, $0 \le j < i \le k$, form a complete set of representatives 
of the equivalence classes of irreducible $K(\mathfrak{sl}_2,k)$-modules. 
We write $M^0$ for $K(\mathfrak{sl}_2,k) = M^{k,0} \cong M^{0,0}$.  
Among those irreducible $M^0$-modules, 
$M^{2j,j}$, $0 \le j  \le \lfloor k/2 \rfloor$, are of $\sigma$-type, 
where $\lfloor k/2 \rfloor$ is the largest integer which does not exceed $k/2$. 
In fact, the automorphism of order $k$ of the fusion algebra of $M^0$ is trivial on 
$M^{i,j}$ if and only if $i = 2j$ for some $0 \le j  \le \lfloor k/2 \rfloor$ 
(Theorem \ref{thm:Zk_sym}). 

Assume that $k \ge 3$, 
and let $M^{0,+}$ and $M^{0,-}$ be the eigenspaces for $\theta$ in 
$M^0$ with eigenvalues $1$ and $-1$, respectively. 
Thus $M^{0,+} = K(\mathfrak{sl}_2,k)^{\la \theta \ra}$. 
The irreducible $M^{0,+}$-modules and the fusion product among them are known 
\cite{JW2017}, \cite{JW2019}. 
Using the results, 
we see that the $\sigma$-type irreducible $M^0$-module $M^{2j,j}$ 
is a direct sum of two irreducible $M^{0,+}$-modules $(M^{2j,j})^\epsilon$ for 
$\epsilon = 0, 1$,   
\[
  M^{2j,j} = (M^{2j,j})^0 \oplus (M^{2j,j})^1,
\] 
where the top level of $(M^{2j,j})^0$ agrees with 
the top level of $M^{2j,j}$. 
Moreover, it follows from the fusion product among $(M^{2j,j})^{\ep}$ for  
$0 \le j \le \lfloor k/2 \rfloor$ and $\ep \in \{ 0,1 \}$ 
(Theorem \ref{thm:fusion_M2jj01} and Remark \ref{rem:fusion_alg_generator})
that
\[
  (M^{2j,j})^{\ep} \mapsto (-1)^{j+\ep}(M^{2j,j})^{\ep}
\]
gives rise to an automorphism of order $2$ of the subalgebra of the fusion algebra of 
$M^{0,+}$ spanned by 
$(M^{2j,j})^{\ep}$ for $0 \le j \le \lfloor k/2 \rfloor$ and $\ep \in \{ 0,1 \}$ 
(Theorem \ref{thm:Z2_sym}).
Therefore, for a vertex operator algebra $V$ containing 
a vertex operator subalgebra $W \cong K(\mathfrak{sl}_2,k)$ such that 
$V$ is a direct sum of $\sigma$-type irreducible $W$-modules, 
we can define an automorphism $\sigma_W$ of the vertex operator algebra $V$ 
by multiplying the elements of the irreducible direct summands 
isomorphic to $(M^{2j,j})^{\ep}$ by $(-1)^{j+\ep}$ (Theorem \ref{thm:sigma-inv}). 
We call $\sigma_W$ the $\sigma$-involution of $V$ associated with $W$.

A lattice vertex operator algebra $V_{\sqrt{2}A_{k-1}}$ associated with 
$\sqrt{2}$ times a root lattice of type $A_{k-1}$ 
contains a vertex operator subalgebra $W \cong K(\mathfrak{sl}_2,k)$. 
It is shown that the $\sigma$-involution associated with $W$ 
coincides with the lift $\theta$ of the $-1$-isometry of $\sqrt{2}A_{k-1}$ 
as an automorphism of $V_{\sqrt{2}A_{k-1}}$ (Theorem \ref{thm:theta-sigma}). 
We also study $\sigma$-involutions of a lattice vertex operator algebra $V_L$  
for a positive definite even lattice $L$ containing a sublattice $N \cong \sqrt{2}A_{k-1}$. 
We show that if $N$ is RSSD in $L$, that is, $2L \subset N + \Ann_L(N)$, 
then $V_L$ is a direct sum of $\sigma$-type irreducible $W$-modules 
((1) of Theorem \ref{thm:N_RSSD_L}). 
Assume further that $L$ has no element of square norm $2$. 
Then the $\sigma$-involution $\sigma_W$ corresponds to 
the RSSD involution $t_N$ of $L$ associated with $N$ 
((2) of Theorem \ref{thm:N_RSSD_L}). 
As to the notion of RSSD involutions, see Section \ref{subsec:RSSD}. 

Let $\nu$ be a fixed point free isometry of $\sqrt{2}A_{k-1}$ of order $k$ 
which corresponds to a Coxeter element of the Weyl group of 
the root system of type $A_{k-1}$. 
Let $\hn \in \Aut (V_{\sqrt{2}A_{k-1}})$ be a lift of $\nu$. 
We determine the centralizer of $\hn$ and the normalizer of $\la \hn \ra$ in 
$\Aut(V_{\sqrt{2}A_{k-1}})$ (Theorem \ref{thm:aut_VNnu}), 
and discuss their structure from a point of view of $\sigma$-involutions 
(Proposition \ref{prop:other_sigma}). 
Furthermore, we show that the automorphism group of the fixed point subalgebra 
$V_{\sqrt{2}A_{k-1}}^{\la \hat{\nu} \ra}$ is isomorphic to the quotient group of 
the normalizer of $\la \hn \ra$ in $\Aut(V_{\sqrt{2}A_{k-1}})$ by $\la \hn \ra$ 
provided that $k$ is an odd prime 
(Theorem \ref{thm:Aut_VNnu}).

The notion of RSSD involutions was introduced in \cite{Griess2011} 
as a generalization of reflections. 
We present an example of lattice vertex operator algebras in which  
the RSSD involutions of the lattice corresponding to $\sigma$-involutions 
are in fact reflections associated with roots, and they generate the Weyl group 
of the root system (Theorem \ref{thm:sec-7-1}). 

In the case $k = 3$, $\sigma$-involutions were previously 
introduced in \cite{Miyamoto2001}, and their properties 
together with interesting examples were studied in \cite{LY2014}.  
In this paper, we consider $\sigma$-involutions 
associated with $K(\mathfrak{sl}_2,k)$ for an arbitrary integer $k \ge 3$. 
We extend and refine the arguments in \cite{LY2014} 
by using a slightly different method to develop  
the theory of $\sigma$-involutions for a general $k$. 
The calculation concerning Virasoro vectors of central charge $4/5$ 
in Griess algebras plays a role in \cite{LY2014}, 
while our method is based on the representation theory 
of the orbifold $K(\mathfrak{sl}_2,k)^{\la \theta \ra}$ \cite{JW2017, JW2019}, 
and a realization of the parafermion vertex operater algebra $K(\mathfrak{sl}_2,k)$ 
in the lattice vertex operator algebra $V_{\sqrt{2}A_{k-1}}$.

$\sigma$-involutions not related to $V_{\sqrt{2}A_{k-1}}$ are of special interest 
as they would give extra symmetries. 
In the case $k = 3$, such a $\sigma$-involution of $V_{K_{12}}^{\la \hat{\nu} \ra}$ 
was studied \cite[Section 5.5]{LY2014}, where $K_{12}$ is the Coxeter-Todd lattice. 
We show that such a $\sigma$-involution also exists in the case $k = 5$, 
see Section \ref{subsec:nonstandard}. 
In fact, the $\sigma$-involution is related to the parafermion vertex operator algebra 
$K(\mathfrak{sl}_2,5)$ contained in $U_{5A}$, 
where $U_{5A}$ denotes the vertex operator algebra $U$ constructed 
in \cite{LYY2005, LYY2007} for the $5A$ case.

This paper is organized as follows. 
Section \ref{sec:preliminaries} is devoted to preliminaries.  
We recall the notion of RSSD involutions. 
Moreover, we review a central extension of 
a positive definite even lattice by a group of order $2$ and
the automorphism group of a lattice vertex operator algebra. 
We collect basic properties of 
the parafermion vertex operator algebra $K(\mathfrak{sl}_2,k)$ as well. 
In Section \ref{sec:sigma-involution}, 
we recall a $\Z_k$ symmetry in the fusion algebra of $K(\mathfrak{sl}_2,k)$, 
and introduce a $\sigma$-involution associated with $K(\mathfrak{sl}_2,k)$. 
In Section \ref{sec:s_VN}, we discuss some $\sigma$-involutions 
of the lattice vertex operator algebra $V_{\sqrt{2}A_{k-1}}$, 
and show how a $\sigma$-involution is related to an RSSD involution. 
In Section \ref{sec:cent_in_aut_VNnu}, we study the centralizer of a lift $\hn$ of 
a fixed point free isometry $\nu$ of the lattice $A_{k-1}$ of order $k$ in 
$\Aut(V_{\sqrt{2}A_{k-1}})$.  
In Section \ref{sec:aut_VNnu}, we determine the automorphism group of 
the fixed point subalgebra of the vertex operator algebra $V_{\sqrt{2}A_{k-1}}$ 
by $\hn$ in the case where $k$ is an odd prime. 
In Section \ref{sec:examples}, we provide some examples of 
$\sigma$-involutions of certain lattice vertex operator algebras. 
We review the irreducible modules and fusion rules for $U_{5A}$ in Appendix. 

\medskip\noindent
\textbf{Acknowledgments.} 
Part of the work was done while the second author was staying at 
Institute of Mathematics, Academia Sinica, Taiwan in February and March, 2020. 
He is grateful to the institute. 
The first author was supported by 
grant AS-IA-107-M02 of Academia Sinica and 
MoST grant 107-2115-M-001-003-MY3 of Taiwan.

\section{Preliminaries}\label{sec:preliminaries}

Let $(L, \la \,\cdot\,,\,\cdot\,\ra)$ be  
a positive definite integral lattice,  
that is, $L$ is a free $\Z$-module of finite rank equipped with 
a positive definite symmetric $\Z$-bilinear form 
$\la \,\cdot\,,\,\cdot\,\ra : L \times L \to \Z$. 
We set $L(n) = \{ \alpha \in L \mid \la \alpha,\alpha \ra = n \}$ 
for a positive integer $n$. An element $\alpha \in L(2)$ is called a root. 
If $\la \alpha, \alpha \ra \in 2\Z$ for any $\alpha \in L$, 
then $L$ is said to be even. 
The isometry group $O(L)$ of $L$ 
is the group of automorphisms of the free $\Z$-module $L$ 
preserving $\la \,\cdot\,,\,\cdot\,\ra$, that is, 
\begin{equation*}
  O(L) = \{g \in \Aut(L) \mid \la g\alpha, g\beta \ra = \la \alpha, \beta \ra 
  \text{ for } \alpha, \beta \in L\}.
\end{equation*}

Let $\Q L$ be the $\Q$-vector space spanned by $L$, 
which may be identified with $\Q \otimes_\Z L$. 
We extend $\la \,\cdot\,,\,\cdot\,\ra$ 
to $\Q L \times \Q L \to \Q$ by $\Q$-linearly. 
The dual lattice $\{ \alpha \in \Q L \mid \la \alpha, L \ra \subset \Z\}$ 
of $L$ is denoted by $L^*$.

\subsection{RSSD involutions}\label{subsec:RSSD}

We recall the notion of RSSD involutions, which was introduced in 
\cite[Section 2.6]{Griess2011} as a generalization of reflections. 
Let $L$ be a positive definite integral lattice. 
For a sublattice $A$ of $L$, set
\[
  \Ann_L(A) = \{ \alpha \in L \mid \la \al, A \ra = 0 \}.
\]
Then $\rank A + \rank( \Ann_L(A)) = \rank L$ and 
\begin{equation}\label{eq:AB_dec}
   A \oplus \Ann_L(A) \subset L \subset L^* \subset A^* \oplus \Ann_L(A)^* \subset \Q L
\end{equation}
with $\Q L = \Q A \oplus \Q \Ann_L(A)$. 
In fact, $\Q \Ann_L(A) = (\Q A)^\perp$ is the orthogonal complement of 
$\Q A$ in $\Q L$ with respect to $\la \,\cdot\,,\,\cdot\,\ra$. 
Define an orthogonal transformation $t_A$ on $\Q L$ by 
\begin{equation*}
  t_A = 
  \begin{cases}
  -1 & \text{on} \ \Q A,\\
  1 & \text{on} \ \Q \Ann_L(A).
  \end{cases}
\end{equation*}

A sublattice $A$ of $L$ is said to be relatively semiselfdual (RSSD) in $L$ 
if $2L \subset A + \Ann_L(A)$ \cite[Definition 2.6.7]{Griess2011},  
see also \cite[Definition 2.5]{GL2011a}. 
We cite Proposition 2.6.8 of \cite{Griess2011} as follows. 

\begin{proposition}\label{prop:tA}
If $A$ is RSSD in $L$, then $t_A(L) = L$. 
\end{proposition}

Indeed, each $\alpha \in L$ is uniquely written as $\alpha = a + b$ 
with $a \in A^*$ and $b \in \Ann_L(A)^*$, and $t_A(\alpha) = - a + b$. 
If $2\alpha = 2a + 2b \in A + \Ann_L(A)$, then $2a \in A$, so $t_A(\alpha) \in L$. 

For an RSSD sublattice $A$ of $L$, the isometry $t_A \in O(L)$ 
of the lattice $L$ is called an RSSD involution 
\cite[Definition 2.6.7]{Griess2011}. 
Let 
\begin{equation*}
  \pr_{A^*} : A^* \oplus \Ann_L(A)^* \to A^*
\end{equation*}
be the projection. 
Then $\pr_{A^*}(\alpha) = a$ for $\alpha = a + b \in L$ with $a \in A^*$ and $b \in \Ann_L(A)^*$.  
Hence the next lemma holds.

\begin{lemma}\label{lem:prA}
$2  \pr_{A^*}(L) \subset A$ if and only if $A$ is RSSD in $L$.
\end{lemma}

\begin{remark}\label{rmk:VL_VAVB}
Let $A$ be a sublattice of a positive definite even lattice $L$,  
and set $B = \Ann_L(A)$. 
Then the vertex operator algebra $V_L$ associated with $L$ 
is a simple current extension of $V_{A \oplus B} = V_A \otimes V_B$.
\end{remark}

\subsection{$\hat{L}$ and its automorphisms}\label{subsec:aut_hatL}

We review a central extension $\hat{L}$ of a lattice $L$ 
by a group of order $2$, see for example \cite[Section 3.8]{MP1995}. 
Let $(L, \la \,\cdot\,,\,\cdot\,\ra)$ be a positive definite even lattice.
Set $\olL = L/2L$, 
which is a vector space over $\Z_2$ with dimension $\rank L$. 
A map $q : \olL \to \Z_2$ is called a quadratic form if the associated map 
$b : \olL \times \olL \to \Z_2$ defined by
\begin{equation}\label{eq:qb}
  b(x,y) = q(x+y) + q(x) + q(y)
\end{equation}
is bilinear. 
This condition implies that $b$ is alternating. 
We have
\begin{equation}\label{eq:q_value}
  q  (a_1 v_1 + \cdots + a_m v_m) = \sum_{i=1}^m a_i q(v_i) 
  +  \sum_{1 \le i < j \le m} a_i a_j b(v_i,v_j)
\end{equation}
for $a_i \in \Z_2$ and $v_i \in \olL$ by \eqref{eq:qb}. 

If $\mu : \olL \to \Z_2$ is a linear map, then 
$q + \mu$ is also a quadratic form with the same associated bilinear form. 
For an alternating bilinear form $b$ on $\olL$, 
there is a quadratic form on $\olL$ whose associated bilinear form is $b$. 
Such a quadratic form is uniquely determined by $b$ 
up to $\Hom(\olL, \Z_2)$. 
On the other hand, for a quadratic form $q$ on $\olL$,  
there is a bilinear form $\vep$ on $\olL$ satisfying 
$q(x) = \vep(x,x)$ for $x \in \olL$. 
Such a bilinear form $\vep$ is uniquely determined by $q$ 
up to an alternating bilinear form on $\olL$. 
For $\alpha$, $\beta \in L$, 
let $\ola = \alpha + 2L$ and $\olb = \beta + 2L \in \olL$. 
If $b : L \times L \to \Z_2$ is a $\Z$-bilinear map, then the radical of $b$ 
contains $2L$. So $b$  induces a bilinear form $\olL \times \olL \to \Z_2$; 
$(\ola, \olb) \mapsto b(\alpha,\beta)$. 
For simplicity of notation, we also write $b$ for the induced bilinear form. 
Then $b(\ola,\olb) = b(\alpha,\beta)$. 
Similarly, we sometimes do not distinguish $\Hom(L,\Z_2)$ and $\Hom(\olL,\Z_2)$.

Consider a map $q : \olL \to \Z_2$ defined by 
\[
  q(\ola) = \frac{1}{2} \la \alpha,\alpha \ra +2\Z \quad \text{for}  \ \alpha \in L. 
\]
The map $q$ is a quadratic form whose associated bilinear form $b$ is
given by $b(\ola,\olb) = \la \alpha,\beta \ra +2\Z$. 
Hence there is a bilinear form $\vep : \olL \times \olL \to \Z_2$ satisfying
\begin{equation}\label{eq:cond_vep}
  \vep(\ola,\ola) = \frac{1}{2} \la \alpha,\alpha \ra +2\Z  
  \quad \text{for} \ \alpha \in L. 
\end{equation}
Such a bilinear form $\vep$ is uniquely determined 
up to an alternating bilinear form on $\olL$. 
The condition \eqref{eq:cond_vep} implies that
\begin{equation}\label{eq:cond2_vep}
    \vep(\ola,\olb) + \vep(\olb,\ola) =\la \alpha,\beta \ra +2\Z
    \quad \text{for} \ \alpha, \beta \in L.
\end{equation}

We fix a bilinear form $\vep : \olL \times \olL \to \Z_2$ satisfying 
\eqref{eq:cond_vep}. 
Let 
\begin{equation}\label{eq:ext_L}
  1 \longrightarrow \la \kappa \ra \longrightarrow \hat{L} 
  \longrightarrow L \longrightarrow 0
\end{equation}
be a central extension of the additive group $L$ 
by a group $\la \kappa \ra$ of order $2$ with $2$-cocyle $\vep$. 
Each element of $\hat{L}$ is uniquely written as $e^\alpha \kappa^a$ 
for $\alpha \in L$ and $a \in \Z_2$, and the product of 
$e^\alpha \kappa^a$ and $e^\beta \kappa^b$ is 
\begin{equation}\label{eq:def-prod}
  e^\alpha \kappa^a e^\beta \kappa^b 
  = e^{\alpha+\beta} \kappa^{\vep(\ola,\olb) + a + b}. 
\end{equation}

Let $\vep' : \olL \times \olL \to \Z_2$ be another bilinear form satisfying 
\eqref{eq:cond_vep}. 
Then there is a quadratic form $\xi$ on $\olL$ such that 
\[
  \xi(\ola + \olb) + \xi(\ola) + \xi(\olb) = \vep(\ola,\olb) + \vep'(\ola,\olb)
  \quad \text{for} \ \alpha, \beta \in L. 
\]
Set $e'^\alpha = e^\alpha \kappa^{\xi(\ola)}$. 
Then the multiplication in $\hat{L}$ is written as
\begin{equation*}
  e'^\alpha \kappa^a e'^\beta \kappa^b 
  = e'^{\alpha + \beta} \kappa^{\vep'(\ola,\olb) + a + b}.
\end{equation*}

For $g \in O(L)$, define a bilinear form $\vep^g : \olL \times \olL \to \Z_2$ 
by $\vep^g(\ola,\olb) = \vep(\overline{g\alpha}, \overline{g\beta})$. 
Then $\vep^g$ also satisfies \eqref{eq:cond_vep}. 
Let 
\[
  b_g =  \vep + \vep^g : \olL \times \olL \to \Z_2; \quad  
  b_g(\ola,\olb) 
  = \vep(\ola,\olb) + \vep(\overline{g\alpha}, \overline{g\beta}).
\]
Then $b_g$ is an alternating bilinear form. 
Hence there is a quadratic form $\eta$ on $\olL$ with associated bilinear 
form $b_g$.
\begin{equation}\label{eq:def-eta}
  \eta(\ola + \olb) + \eta(\ola) + \eta(\olb) = b_g(\ola,\olb) 
  \quad \text{for} \ \alpha, \beta \in L.
\end{equation}

Define a map $\hat{g} : \hat{L} \to \hat{L}$ by 
\begin{equation}\label{def-hatg}
  \hat{g}(e^\alpha \kappa^a) = e^{g \alpha}\kappa^{\eta(\ola)+a}.
\end{equation}
Then $\hat{g}$ is an automorphism of the group $\hat{L}$. 
The automorphism $\hat{g}$ depends on the choice of $\eta$ 
satisfying \eqref{eq:def-eta}. 
The inverse of $\hat{g}$ is given by  
\[
  \hat{g}^{-1}(e^\alpha \kappa^a) 
  = e^{g^{-1}\alpha} \kappa^{\eta(\overline{g^{-1}\alpha}) + a}.
\]

For $f \in O(L)$, let $\xi$ be a quadratic form on $\olL$ 
with associated bilinear form $b_f = \vep + \vep^f$, 
and define a map $\hat{f} : \olL \to \olL$ by 
$\hat{f}(e^\alpha \kappa^a) = e^{f \alpha}\kappa^{\xi(\ola)+a}$. 
Then
\begin{equation*}
  \hat{f}\hat{g}(e^\alpha \kappa^a) 
  = e^{fg\alpha} \kappa^{\eta(\ola) + \xi(\overline{g \alpha}) + a}.
\end{equation*}
Note that the map $\eta + \xi^g : \olL \to \Z_2$; 
$\ola \mapsto \eta(\ola) + \xi(\overline{g \alpha})$ is a quadratic form 
with associated bilinear form $b_{fg} = \vep + \vep^{fg}$.

Let $\eta : \olL \to \Z_2$ be a map. 
Then a map $\phi : \hat{L} \to \hat{L}$ defined by 
$\phi(e^\alpha \kappa^a) = e^{g \alpha}\kappa^{\eta(\ola)+a}$ 
is an automorphism of the group $\hat{L}$ 
if and only if $\eta$ is a quadratic form 
with associated bilinear form $b_g$. 
In such a case, $\phi$ is called a lift of $g$. 

Denote by $O(\hat{L})$ the group of automorphisms $\phi$ of $\hat{L}$ 
such that $\phi$ is a lift of some $g \in O(L)$. 
Let $\varphi : O(\hat{L}) \to O(L)$ be a group homomorphism 
defined by $\varphi(\phi) = g$. 
Then we have a exact sequence
\begin{equation}\label{eq:ext_OhatL}
  1 \longrightarrow \Hom(L,\Z_2) \longrightarrow O(\hat{L}) 
  \stackrel{\varphi}{\longrightarrow} O(L) \longrightarrow 1
\end{equation}
of groups, where 
$\Hom(L,\Z_2) \to O(\hat{L}); \lambda \mapsto \hat{\lambda}$ 
is defined by 
$\hat{\lambda}(e^\alpha \kappa^a) = e^\alpha \kappa^{\lambda(\ola) + a}$, 
see \cite[Section 2.4]{DN1999} and \cite[Section 3.8, Proposition 6]{MP1995}, 
see also \cite[Proposition 5.4.1]{FLM1988}.

Let $\alpha \in L$, and choose a positive integer $n$ such that 
$g^n \alpha = \alpha$. 
Then
\begin{equation}\label{eq:eta-bg}
  \eta(\ola + \overline{g\alpha} + \cdots + \overline{g^{n-1}\alpha}) 
  = \sum_{i=0}^{n-1} \eta(\overline{g^i\alpha}) 
  + \sum_{0 \le i < j \le n-1} b_g(\overline{g^i\alpha}, \overline{g^j\alpha})
\end{equation}
by \eqref{eq:q_value}. 
Since $b_g = \vep + \vep^g$, we have 
\begin{equation}\label{eq:sum_bg}
  \sum_{0 \le i < j \le n-1} b_g(\overline{g^i\alpha}, \overline{g^j\alpha}) 
  = \sum_{1 \le j \le n-1} \la \alpha, g^j \alpha \ra + 2\Z
\end{equation}
by \eqref{eq:cond2_vep}. 
Moreover, 
$\la \alpha, g^j \alpha \ra = \la \alpha, g^{n-j} \alpha \ra$ as $g^n \alpha = \alpha$.  
Therefore, 
\begin{equation*}
  \eta(\ola + \overline{g\alpha} + \cdots + \overline{g^{n-1}\alpha}) = 
  \begin{cases}
  \sum_{i=0}^{n-1} \eta(\overline{g^i\alpha}) 
  & \text{ if } n \text{ is odd},\\
  \sum_{i=0}^{n-1} \eta(\overline{g^i\alpha}) + \la \alpha, g^{n/2} \alpha \ra + 2\Z
  & \text{ if } n \text{ is even}.
  \end{cases}
\end{equation*}
Note that $\hat{g}^n (e^\alpha) = e^{g^n \alpha} \kappa^\delta$ with 
$\delta = \sum_{i=0}^{n-1} \eta(\overline{g^i\alpha})$, 
see \cite[Lemma 12.1]{Borcherds1992}, \cite[Propositions 7.1 and 7.2]{vEMS2020}.

When $g$ is fixed point free on $L$ of order $n$,  
we have $(1 + g + \cdots + g^{n-1}) \alpha = 0$ and then 
$\sum_{1 \le j \le n-1} \la \alpha, g^j \alpha \ra = - \la \alpha, \alpha \ra \in 2\Z$. 
Therefore, $\sum_{i=0}^{n-1} \eta(\overline{g^i\alpha}) = 0$ 
by \eqref{eq:eta-bg} and \eqref{eq:sum_bg}. 
Thus the following proposition holds. 

\begin{proposition}\label{prop:order-hatg}
Suppose $g \in O(L)$ is fixed point free on $L$ of order an integer $n \ge 2$. 
Then the order of $\hat{g} \in O(\hat{L})$ is also $n$. 
\end{proposition}

The next simple lemma will be used later.

\begin{lemma}\label{lem:div_g_Hom}
Let $g : L \to L$ be an automorphism of a finite rank free $\Z$-module $L$ 
of order an odd integer $p \ge 3$. 
Assume that $g$ is fixed point free on $L$. 
Then the map 
\[
  \Hom(\olL,\Z_2) \to \Hom(\olL,\Z_2); \quad \mu \mapsto \mu + \mu^g
\] 
is a $\Z_2$-linear isomorphism, 
where $\mu^g$ is defined by $\mu^g(\ola) = \mu(\overline{g\alpha})$ 
for $\ola = \alpha + 2L \in \olL$.
\end{lemma}

\begin{proof}
Suppose $\mu + \mu^g = 0$. 
Then $\mu(\overline{(1+g+\cdots+g^{p-1})\alpha}) = p \mu(\ola)$.  
Since $g$ is fixed point free of order $p$, we have 
$(1+g+\cdots+g^{p-1})\alpha = 0$. 
Thus $\mu(\ola) = 0$ as $p$ is an odd integer. 
Hence the kernel of the linear map $\mu \mapsto \mu + \mu^g$ is trivial,  
so the assertion holds.
\end{proof}

We consider the relationship between the normalizer $N_{O(L)}(\la g \ra)$ 
of $\la g \ra$ (resp. the centralizer $C_{O(L)}(g)$ of $g$) in $O(L)$ 
and the normalizer $N_{O(\hat{L})}(\la \hat{g} \ra)$ of $\la \hat{g} \ra$ 
(resp. the centralizer $C_{O(\hat{L})}(\hat{g})$ of $\hat{g}$) in $O(\hat{L})$.

\begin{proposition}\label{prop:OL_OhatL}
Assume that $g \in O(L)$ is fixed point free on $L$ 
of order an odd integer $p \ge 3$. 
Let $\hat{g} \in O(\hat{L})$ be a lift of $g$. 

\textup{(1)} 
For each $f \in N_{O(L)}(\la g \ra)$, there is a unique lift $\hat{f}$ of $f$ 
in $N_{O(\hat{L})}(\la \hat{g} \ra)$, 
and the restriction of $\varphi$ in  
\eqref{eq:ext_OhatL} to $N_{O(\hat{L})}(\la \hat{g} \ra)$ 
is an isomorphism from $N_{O(\hat{L})}(\la \hat{g} \ra)$ 
to $N_{O(L)}(\la g \ra)$. 

\textup{(2)}
For each $f \in C_{O(L)}(g)$, there is a unique lift $\hat{f}$ of $f$ 
in $C_{O(\hat{L})}(\hat{g})$,  
and the restriction of $\varphi$ in  
\eqref{eq:ext_OhatL} to $C_{O(\hat{L})}(\hat{g})$ 
is an isomorphism from $C_{O(\hat{L})}(\hat{g})$ 
to $C_{O(L)}(g)$. 
\end{proposition}

\begin{proof}
Let $\eta$ be a quadratic form on $\olL$ with the associated bilinear form 
$b_g = \vep + \vep^g$ such that 
$\hat{g}(e^\alpha \kappa^a) = e^{g \alpha}\kappa^{\eta(\ola)+a}$. 
Let $f \in N_{O(L)}(\la g \ra)$. 
Then $f^{-1}gf = g^m$ for some $1 \le m \le p-1$, where $m$ and $p$ are 
coprime. 
Let $\hat{f} \in O(\hat{L})$ be a lift of $f$, 
and let $\xi$ be a quadratic form on $\olL$ with associated bilinear form 
$b_f = \vep + \vep^f$ such that 
$\hat{f}(e^\alpha \kappa^a) = e^{f \alpha} \kappa^{\xi(\ola) + a}$.  

Set $h = g^m$. Since $\hat{g}^m$ is a lift of $h$, 
we have $\hat{g}^m(e^\alpha \kappa^a) = e^{h\alpha} \kappa^{\zeta(\ola) + a}$ 
for some quadratic form $\zeta$ on $\olL$ with associated bilinear form 
$b_{h} = \vep + \vep^{h}$.
We also have $\hat{f}^{-1}\hat{g}\hat{f}(e^\alpha \kappa^a) 
= e^{h\alpha} \kappa^{\xi(\overline{h\alpha}) + \eta(\overline{f \alpha})
+ \xi(\ola) + a}$. 
Both $\hat{g}^m$ and $\hat{f}^{-1}\hat{g}\hat{f}$ are lifts of $h$, 
so $\zeta + \eta^f + \xi + \xi^h = \lambda$ for some 
$\lambda \in \Hom(\olL,\Z_2)$ by \eqref{eq:ext_OhatL}, where 
$\eta^f$ and $\xi^h$ are defined by $\eta^f(\ola) = \eta(\overline{f \alpha})$ 
and $\xi^h(\ola) = \xi(\overline{h\alpha})$, respectively. 
There is a unique $\mu \in \Hom(\olL,\Z_2)$ such that 
$\lambda = \mu + \mu^h$ by Lemma \ref{lem:div_g_Hom}, 
where $\mu^h$ is defined by $\mu^h(\ola) = \mu(\overline{h\alpha})$. 
Set $\phi = \hat{f}\hat{\mu}$. 
Then $\phi(e^\alpha \kappa^a) = e^{f \alpha} \kappa^{(\xi+\mu)(\ola) + a}$,  
and  
$\phi^{-1} \hat{g} \phi(e^\alpha \kappa^a) = \hat{g}^m(e^\alpha \kappa^a)$. 
Therefore, $\phi$ is a unique lift of $f$ satisfying 
$\phi^{-1}\hat{g}\phi = \hat{g}^m$. 
Thus the assertion (1) holds. 

The assertion (2) follows from the above argument with $m = 1$. 
\end{proof}

If $g$ is the $-1$-isometry of $L$, then $g$ is fixed point free of order $2$. 
In this case,  we have $b_g = 0$. 
Define $\theta \in O(\hat{L})$ by 
\begin{equation*}
  \theta(e^\alpha \kappa^a) = e^{-\alpha} \kappa^a
\end{equation*}
for $\alpha \in L$ and $a \in \Z_2$. 
Then $\theta$ is a lift of $-1$ with order $2$. 
Moreover, $\phi \theta = \theta \phi$ for any $\phi \in O(\hat{L})$. 
For any fixed point free $f \in O(L)$ of order an odd integer $p \ge 3$, 
we have $-f$ is fixed point free on $L$ of order $2p$, and $\hat{f} \theta$ 
is a lift of $-f$.

\subsection{Automorphisms of $V_L$}\label{subsec:aut_VL} 

Let $(L, \la \,\cdot\,,\,\cdot\,\ra)$ be a positive definite even lattice. 
The structure of the automorphism group $\Aut(V_L)$ 
of the vertex operator algebra $V_L$ associated with $L$ was studied in 
\cite[Section 2]{DN1999}. 
Let $\h = \C \otimes_\Z L$. We extend $\la \,\cdot\,,\,\cdot\,\ra$ to 
$\h \times \h \to \C$ by $\C$-linearly. 
Let $\hat{\h} = \h \otimes \C[t,t^{-1}] \oplus \C c$ 
be the corresponding affine Lie algebra. 
Let $M(1)$ be the $\hat{\h}$-module induced from 
an $\h \otimes \C[t] \oplus \C c$-module $\C$, 
where $\h \otimes \C[t]$ acts trivially on $\C$ and $c$ acts as $1$ on $\C$. 
We denote by $h(n)$ the action of $h \otimes t^n$ on $M(1)$ for $h \in \h$. 

Let $\C[L]_\vep = \C[\hat{L}]/(\kappa+1)\C[\hat{L}]$ be the quotient algebra 
of the group algebra $\C[\hat{L}]$ of $\hat{L}$ by the ideal generated 
by $\kappa+1$. 
We use the same symbol $e^\alpha$ to denote the image of $e^\alpha \in \hat{L}$ 
in $\C[L]_\vep$. 
Then $\hat{g} \in O(\hat{L})$ defined in \eqref{def-hatg} acts on $\C[L]_\vep$ by 
$\hat{g}(e^\alpha) = (-1)^{\eta(\ola)} e^{g\alpha}$. 

The vertex operator algebra $V_L$ is defined to be 
$V_L = M(1) \otimes \C[L]_\vep$ 
as a vector space \cite[Chapter 8]{FLM1988}, \cite[Section 6.4]{LL2004}. 
Any $\hat{g} \in O(\hat{L})$ acts on $V_L$ by 
\begin{equation*}
  \hat{g} (h_1(-n_1) \cdots h_r(-n_r) \otimes e^\alpha) 
  = g(h_1)(-n_1) \cdots g(h_r)(-n_r) \otimes \hat{g}(e^\alpha)
\end{equation*}
for $h_i \in \h$, $n_i > 0$, and $\alpha \in L$. 
Set
\begin{equation*}
  N(V_L) = \{ \exp(v_{(0)}) \mid v \in (V_L)_1 \}, 
\end{equation*}
which is a normal subgroup of $\Aut(V_L)$ as 
$\phi \exp(v_{(0)}) \phi^{-1} = \exp((\phi v)_{(0)})$ for $\phi \in \Aut(V_L)$. 
It is known \cite[Theorem 2.1]{DN1999} that
\[
  \Aut(V_L) = N(V_L) O(\hat{L}) \quad \text{with} \quad 
  N(V_L) \cap O(\hat{L}) \supset \Hom(L,\Z_2).
\]

Now, assume that $L(2) = \varnothing$,  
where $L(2) = \{ \alpha \in L \mid \la \alpha,\alpha \ra = 2 \}$. 
Then $(V_L)_1 = \{ h(-1)\1 \mid h \in \h \}$, and 
\begin{equation}\label{eq:N_cap_OhatL}
  N(V_L) = \{ \exp(h(0)) \mid h \in \h \} \quad \text{with} \quad 
  N(V_L) \cap O(\hat{L}) = \Hom(L,\Z_2). 
\end{equation}
Thus $\varphi : O(\hat{L}) \to O(L)$ in \eqref{eq:ext_OhatL} can be extended to 
$\varphi : \Aut(V_L) \to O(L)$ with kernel $\Ker \varphi = N(V_L)$,   
and we obtain an exact sequence
\begin{equation}\label{eq:ex_seq_Aut_VL}
  1 \longrightarrow N(V_L) \longrightarrow \Aut(V_L) 
  \stackrel{\varphi}{\longrightarrow} O(L) \longrightarrow 1 
\end{equation}
of groups, see \cite[Remark 5.14]{LY2014}. 
The automorphism $\exp(h(0))$ acts on $V_L$ by 
\begin{equation*}
  \exp(h(0)) (h_1(-n_1) \cdots h_r(-n_r) \otimes e^\alpha) 
  = \exp(\la h,\alpha \ra) h_1(-n_1) \cdots h_r(-n_r) \otimes e^\alpha.
\end{equation*}

The next theorem is a refinement of \cite[Theorem 5.15]{LY2014}. 

\begin{theorem}\label{thm:C_AutVL_hatg}
Let $L$ be a positive definite even lattice with $L(2) = \varnothing$. 
Let $g \in O(L)$ be fixed point free on $L$ of order an odd integer $p \ge 3$, 
and let $\hat{g} \in O(\hat{L})$ be a lift of $g$. 

\textup{(1)} 
$C_{\Aut(V_L)}(\hat{g}) = C_{N(V_L)}(\hat{g}) : C_{O(\hat{L})}(\hat{g})$ is a split extension of 
$C_{N(V_L)}(\hat{g})$ by $C_{O(\hat{L})}(\hat{g})$.

\textup{(2)} 
$N_{\Aut(V_L)}(\la \hat{g} \ra) = C_{N(V_L)}(\hat{g}) : N_{O(\hat{L})}(\la \hat{g} \ra)$ 
is a split extension of 
$C_{N(V_L)}(\hat{g})$ by $N_{O(\hat{L})}(\la \hat{g} \ra)$.

\textup{(3)} 
$C_{N(V_L)}(\hat{g}) = \{ \exp(h(0)) \mid h \in 2\pi\sqrt{-1}((1-g)L)^* \}$.

\textup{(4)} 
$C_{N(V_L)}(\hat{g}) \cong \Hom(L/(1-g)L,\Z_p)$, 
where the left-hand side is a multiplicative group and the right-hand side is an additive group. 
The isomorphism is given by assigning 
$\exp(2\pi\sqrt{-1} \gamma(0)) \in C_{N(V_L)}(\hat{g})$ with $\gamma \in ((1-g)L)^*$ to the $\Z$-module homomorphism
\[
  L/(1-g)L \to \Z_p; \quad 
  \alpha + (1-g)L \mapsto p \la \gamma, \alpha + (1-g)L \ra + p\Z
\]
for $\alpha \in L$.
\end{theorem}

\begin{proof}
The assertion (1) follows from (2) of Proposition \ref{prop:OL_OhatL}, 
\eqref{eq:N_cap_OhatL}, and \eqref{eq:ex_seq_Aut_VL}. 

Let $\phi \in N_{\Aut(V_L)}(\la \hat{g} \ra)$. 
Then $\phi \hat{g} \phi^{-1} = \hat{g}^m$  
for some $1 \le m \le p-1$, where $m$ and $p$ are coprime. 
Set $f = \varphi(\phi)$. Then $fgf^{-1} = g^m$. 
As shown in the proof of Proposition \ref{prop:OL_OhatL}, 
there is a unique lift $\hat{f}$ of $f$ satisfying $\hat{f}\hat{g}\hat{f}^{-1} = \hat{g}^m$. 
Then $\hat{f} \in N_{O(\hat{L})}(\la \hat{g} \ra)$  
and $\hat{f}^{-1} \phi \in C_{\Aut(V_L)}(\hat{g})$, so we have  
$\phi \in C_{N(V_L)}(\hat{g}) N_{O(\hat{L})}(\la \hat{g} \ra)$ by the assertion (1). 
Thus the assertion (2) holds by (1) of Proposition \ref{prop:OL_OhatL}. 

For $h \in \h$, we have $\hat{g} \exp(h(0)) \hat{g}^{-1} = \exp((gh)(0))$. 
Hence $\exp(h(0)) \in C_{N(V_L)}(\hat{g})$ if and only if 
$\la gh - h, \alpha \ra \in 2\pi\sqrt{-1}\Z$ for all $\alpha \in L$. 
Since $g \in O(L)$, 
the condition on $h$ is equivalent to that $\la h, (1-g)\alpha \ra \in 2\pi\sqrt{-1}\Z$ 
for all $\alpha \in L$. Thus the assertion (3) holds. 

For $\alpha \in L$, we have 
$\alpha \equiv g^i \alpha \pmod{(1-g)L}$, $1 \le i \le p-1$. 
Since $g$ is fixed point free of order $p$, 
we also have $1 + g + \cdots + g^{p-1} = 0$ on $L$. 
Hence $p\alpha \equiv 0 \pmod{(1-g)L}$, so $L \supset (1-g)L \supset pL$. 
Thus $L^* \subset ((1-g)L)^* \subset \frac{1}{p}L^*$. 

For $\gamma \in L^*$, define $\chi_\gamma \in \Hom(L,\Z_p)$ by 
$\chi_\gamma(\alpha) = \la \gamma, \alpha \ra + p\Z$. 
Then the map $\chi : L^* \to \Hom(L,\Z_p); \gamma \mapsto \chi_\gamma$ 
is a surjective homomorphism of $\Z$-modules. 
The kernal of $\chi$ is $pL^*$. 
Moreover, $\Ker \chi_\gamma \supset (1-g)L$ if and only if $\gamma \in p((1-g)L)^*$. 
Thus $\gamma \mapsto \chi_{p\gamma}$ gives an isomorphism 
\[
  ((1-g)L)^*/L^* \to \Hom(L/(1-g)L,\Z_p)
\]
of additive groups. 
Since $\gamma \mapsto \exp(2\pi\sqrt{-1} \gamma(0))$ induces an isomorphism
\[
  ((1-g)L)^*/L^* \to C_{N(V_L)}(\hat{g})
\]
from an additive group to a multiplicative group, the assertion (4) holds. 
\end{proof}

\begin{remark}\label{rmk:C_AutVL_hatg}
The assumption that the order of $g$ is odd is not necessary 
for the assertions (3) and (4) of Theorem \ref{thm:C_AutVL_hatg}.
\end{remark}

If $p$ is an odd prime in the above theorem, 
then $\rank L$ is divisible by $p-1$, and $L/(1-g)L \cong p^m$ 
is elementary abelian of order $p^m$ 
by \cite[Lemma A.1]{GL2011b}, where $m = \rank L/(p-1)$.
Hence the following corollary holds by Proposition \ref{prop:OL_OhatL}.

\begin{corollary}\label{cor:C_AutVL_hatg}
Assume that $p$ is an odd prime in the above theorem. 
Then $C_{\Aut(V_L)}(\hat{g}) \cong p^{m} :  C_{O(L)}(g)$ and 
$N_{\Aut(V_L)}(\la \hat{g} \ra) \cong p^{m} : N_{O(L)}(\la g \ra)$, 
where $m = \rank L/(p-1)$.
\end{corollary}

\subsection{Parafermion vertex operator algebra $K(\mathfrak{sl}_2,k)$}
\label{subsec:paraf_VOA}

We collect basic properties of 
the parafermion vertex operator algebra $K(\mathfrak{sl}_2,k)$ 
associated with $\mathfrak{sl}_2$ and a positive integer $k$, 
see \cite{ALY2014, ALY2019, DLWY2010, DLY2009, DW2016} for details.  
If $k=1$, then $K(\mathfrak{sl}_2,k)$ reduces to  
the trivial vertex operator algebra $\C\1$. 
We assume that $k \ge 2$. 

(1) $M^0 = K(\mathfrak{sl}_2,k)$ is a simple, self-dual, rational, and
$C_2$-cofinite vertex operator algebra of CFT-type with central charge 
$2(k-1)/(k+2)$.

(2) $K(\mathfrak{sl}_2,2) \cong L(1/2,0)$ is 
the simple Virasoro vertex operator algebra of central charge $1/2$. 
If $k \ge 3$, there is, up to a scalar multiple, a unique primary vector $W^3$ 
of weight $3$, and $K(\mathfrak{sl}_2,k)$ is generated by the vector $W^3$  
as a vertex operator algebra.  

(3) $M^{i, j}$, $0 \le i \le k$, $0 \le j < k$, are irreducible $M^0$-modules with
\begin{equation}\label{eq:isom_Mij}
  M^{i, j} \cong M^{k-i, j-i}, 
\end{equation}
and $M^{i, j}$, $0 \le j < i \le k$, form a complete set of 
representatives of the equivalence classes of irreducible $M^0$-modules. 
The index $j$ is considered to be modulo $k$. 
Note that $M^0 = M^{k,0} = M^{0,0}$.

(4) The top level of $M^{i, j}$ is one dimensional with the conformal weight  
\begin{equation}\label{conf_wt_Mij}
  h(M^{i, j}) = \frac{1}{2k(k+2)}\Big( k(i-2j) - (i-2j)^2 + 2k(i-j+1)j \Big)
\end{equation}
for $0 \le j \le i \le k$. 

(5) The fusion product of irreducible $M^0$-modules is given by 
\begin{equation}\label{eq:paraf_fusion}
  M^{i_1, j_1} \boxtimes_{M^0} M^{i_2, j_2} 
  = \sum_{r \in R(i_1,i_2)} M^{r, (2j_1 - i_1 + 2j_2 - i_2 + r)/2}, 
\end{equation}
where $R(i_1,i_2)$ is the set of integers $r$ such that
\begin{equation*}
  |i_1-i_2| \le r \le \min \{i_1+i_2, 2k - i_1 - i_2\}, \quad i_1+i_2+r \in 2\Z.
\end{equation*} 

Let $M^j = M^{k, j}$ for $0 \le j \le k-1$, which are the simple current $M^0$-modules. 
In fact, 
\begin{equation}\label{eq:paraf_sc}
  M^p \boxtimes_{M^0} M^{i, j} = M^{i, j+p}. 
\end{equation}

(6) If $k \ge 3$, the automorphism group $\Aut(M^0)$ of $M^0$ is 
generated by an involution $\theta$ such that   
$\theta(W^3) = - W^3$, and 
\begin{equation}\label{eq:act_theta}
  M^{i, j} \circ \theta \cong M^{i, i-j}.
\end{equation}

Now, assume that $k \ge 3$. 
An irreducible $M^0$-module $M^{i, j}$ is said to be $\theta$-stable 
if $M^{i, j} \circ \theta \cong M^{i, j}$. 
We see from \eqref{eq:isom_Mij} and \eqref{eq:act_theta} that 
$M^{i, j}$ is $\theta$-stable if and only if $i = 2j$ for 
$0 \le j  \le \lfloor k/2 \rfloor$, or $k$ is even and $(i,j) = (k/2,0)$, 
where $\lfloor k/2 \rfloor$ is the largest integer which does not exceed $k/2$. 
The irreducible $M^0$-modules 
$M^{2j, j}$, $0 \le j  \le \lfloor k/2 \rfloor$, are said to be of $\sigma$-type. 
An $M^0$-module is said to be of $\sigma$-type 
if it is a direct sum of irreducible $M^0$-modules of $\sigma$-type.

\section{$\sigma$-involutions}\label{sec:sigma-involution}

In this section, we introduce a $\sigma$-involution associated with 
the parafermion vertex operator algebra $K(\mathfrak{sl}_2,k)$. 

\subsection{Automorphism $\tau_W$}\label{subsec:aut-tW}

First of all, we recall a $\Z_k$ symmetry in the fusion algebra of 
$M^0 = K(\mathfrak{sl}_2,k)$ for $k \ge 2$ \cite{ZF1985}.  
For $0 \le i \le k$ and $0 \le j < k$, 
let $0 \le l < 2k$ be such that $l \equiv i - 2j \pmod{2k}$, and set  
\begin{equation*}
  \widetilde{M}^{i,l} = M^{i,j}. 
\end{equation*}
Then \eqref{eq:paraf_fusion} can be written as 
\begin{equation}\label{eq:paraf_fusion-new}
  \widetilde{M}^{i_1,l_1} \boxtimes_{M^0} \widetilde{M}^{i_2,l_2} 
  = \sum_{r \in R(i_1,i_2)} \widetilde{M}^{r,l_1+l_2}.
\end{equation}
Moreover, 
\begin{gather}
  \widetilde{M}^{i,l} \cong \widetilde{M}^{k-i,k+l},\label{eq:isom_Mij_new}\\
  \widetilde{M}^{i,l} \circ \theta \cong \widetilde{M}^{i,-l}\label{eq:act_theta_new}
\end{gather}
by \eqref{eq:isom_Mij} and \eqref{eq:act_theta}, respectively. 

The following theorem is known,  
see \cite[Section 4]{YY2019} and \cite{ZF1985}. 

\begin{theorem}\label{thm:Zk_sym}
Let $k \ge 2$ be an integer. Then a map defined by 
\[
  \widetilde{M}^{i,l} \mapsto \zeta_k^l \widetilde{M}^{i,l} \quad \text{for} \ 
  0 \le i \le k \ \text{and} \ 0 \le l < 2k \ \text{with} \ i \equiv l \pmod{2}
\]
is compatible with \eqref{eq:paraf_fusion-new} and \eqref{eq:isom_Mij_new},  
and it induces an automorphism of the fusion algebra of $M^0$ of order $k$, 
where $M^0 = K(\mathfrak{sl}_2,k)$ and $\zeta_k = \exp(2\pi\sqrt{-1}/k)$. 
\end{theorem}

The next lemma is a consequence of \eqref{eq:isom_Mij_new}. 

\begin{lemma}\label{lem:isom_M0_new}
$\widetilde{M}^{i,l}$ for $0 \le i \le k$ and $0 \le l < k$ with 
$i \equiv l \pmod{2}$ form a complete set of representatives of 
the equivalence classes of irreducible $M^0$-modules.
\end{lemma}

Let $(V,Y,\1,\omega)$ be a vertex operator algebra, and assume that 
it contains a vertex operator subalgebra $W$ isomorphic to $K(\mathfrak{sl}_2,k)$. 
Then $V$ is a direct sum of irreducible $W$-modules 
by \cite[Theorem 4.5]{ABD2004}.  
For $0 \le l < k$, 
we denote by $V_W[l]$ the sum of all irreducible $W$-submodules 
of $V$ isomorphic to $\widetilde{M}^{i,l}$ for some $0 \le i \le k$ 
with $i \equiv l \pmod{2}$. 
Then 
\begin{equation}\label{eq:dec_V}
  V = \bigoplus_{l=0}^{k-1} V_W[l]
\end{equation}
by Lemma \ref{lem:isom_M0_new}. 
In view of \eqref{eq:isom_Mij_new}, we may consider $l$ for $V_W[l]$ to be 
modulo $k$. Then \eqref{eq:paraf_fusion-new} implies that 
\begin{equation}\label{eq:u_nv}
  Y(u,x)v \in V_W[l_1 + l_2][[x,x^{-1}]]
\end{equation}
for $u \in V_W[l_1]$ and $v \in V_W[l_2]$. 
The following theorem holds by \eqref{eq:dec_V} and \eqref{eq:u_nv}.

\begin{theorem}\label{thm:tau-autom}
Let $V$ be a vertex operator algebra containing 
a vertex operator subalgebra $W \cong K(\mathfrak{sl}_2,k)$  
for an integer $k \ge 2$. Then 
a linear map $\tau_W$ on $V$ defined by 
\[
  \tau_W(v) = \zeta_k^l v \quad \text{for} \ v \in V_W[l],  0 \le l < k
\]
is an automorphism of the vertex operator algebra $V$, 
where $\zeta_k = \exp(2\pi\sqrt{-1}/k)$.
\end{theorem}

The fixed point subalgebra $V^{\la \tau_W \ra}$ of $V$ by the automorphism 
$\tau_W$ is $V_W[0]$. 
Note that $V$ is a $\sigma$-type $W$-module if and only if $V = V_W[0]$.

\subsection{$M^{2j, j}$ as $K(\mathfrak{sl}_2,k)^{\la \theta \ra}$-module}
\label{subse:def_sigma}

Let $k \ge 3$. 
Then $\Aut(K(\mathfrak{sl}_2,k))$ is generated by an involution $\theta$. 
The fixed point subalgebra $K(\mathfrak{sl}_2,k)^{\la \theta \ra}$ of 
$K(\mathfrak{sl}_2,k)$ by the involution $\theta$ 
is simple, self-dual, rational, $C_2$-cofinite, and of CFT-type 
by \cite{CM2016} and \cite{Miyamoto2015}. 
The representation theory of $K(\mathfrak{sl}_2,k)^{\la \theta \ra}$ 
was studied in detail \cite{JW2017}, \cite{JW2019}. 
In fact, the irreducible modules and their highest weight vectors 
are obtained in \cite{JW2017}, and the fusion rules are determined in \cite{JW2019}. 
We use those results for $\sigma$-type irreducible $K(\mathfrak{sl}_2,k)$-modules. 

Let $M^{0,+}$ and $M^{0,-}$ be the eigenspaces with eigenvalues 
$1$ and $-1$ for $\theta$ in $M^0 = K(\mathfrak{sl}_2,k)$, respectively. 
Thus $M^{0,+} = K(\mathfrak{sl}_2,k)^{\la \theta \ra}$, and 
$M^{0,-}$ is the irreducible $M^{0,+}$-module generared by the weight $3$ 
primary vector $W^3$. 
We see from \cite[Proposition 3.14]{JW2017} that 
for $0 \le j \le \lfloor k/2 \rfloor$, a $\sigma$-type irreducible 
$M^0$-module $M^{2j,j}$ is a direct sum
\begin{equation*}
  M^{2j,j} = (M^{2j,j})^0 \oplus (M^{2j,j})^1
\end{equation*} 
of two irreducible $M^{0,+}$-modules, 
where the conformal weight of $(M^{2j,j})^0$ is
\begin{equation*}
  h((M^{2j,j})^0) = h(M^{2j,j}) = \frac{j(j+1)}{k+2}, 
\end{equation*}
and 
$h((M^{2j,j})^1) - h((M^{2j,j})^0) = 3$, $1$, or $2$ according as 
$j = 0$, $1 \le j < k/2$, or $k$ is even and $j = k/2$, respectively. 
Note that $(M^{0,0})^0 = M^{0,+}$ and $(M^{0,0})^1 = M^{0,-}$. 

The following theorem is taken from Theorem 5.1 of \cite{JW2019}. 
\begin{theorem}\label{thm:fusion_M2jj01}
Let $k \ge 3$ be an integer.

\textup{(1)} 
For $0 \le j \le \lfloor k/2 \rfloor$ and $\ep_1, \ep_2 \in \{ 0,1 \}$,
\begin{equation}\label{eq:fusion-1}
  (M^{0,0})^{\ep_1} \boxtimes_{M^{0,+}} (M^{2j,j})^{\ep_2} 
  = (M^{2j,j})^{\ep_1 + \ep_2},
\end{equation}
where $\ep_1 + \ep_2$ is considered to be modulo $2$. 

\textup{(2)} 
If $k \ge 4$, then
\begin{equation}\label{eq:fusion-2}
  (M^{2,1})^0 \boxtimes_{M^{0,+}} (M^{2j,j})^0
  = (M^{2(j-1), j-1})^0 + (M^{2j,j})^1 + (M^{2(j+1), j+1})^0
\end{equation}
for $1 \le j \le \lfloor k/2 \rfloor - 1$.

\textup{(3)}
For $j = \lfloor k/2 \rfloor$, 
\begin{equation}\label{fusion-3_o}
  (M^{2,1})^0 \boxtimes_{M^{0,+}} 
  (M^{2j, j})^0 
  = (M^{2(j-1), j-1})^0 
  + (M^{2j, j})^1
\end{equation}
if $k$ is odd, and
\begin{equation}\label{fusion-3_e}
  (M^{2,1})^0 \boxtimes_{M^{0,+}} (M^{k,k/2})^0 
  = (M^{k-2,k/2-1})^0
\end{equation}
if $k$ is even.
\end{theorem}

\begin{remark}\label{rem:fusion_alg_generator}
Since the fusion product of irreducible $M^{0,+}$-modules 
is commutative and associative \cite[Theorem 3.7]{Huang2005}, 
the above equations \eqref{eq:fusion-1}, \eqref{eq:fusion-2}, \eqref{fusion-3_o}, 
and \eqref{fusion-3_e} determine the fusion product among $(M^{2j,j})^{\ep}$ 
for all $0 \le j \le \lfloor k/2 \rfloor$ and $\ep \in \{ 0,1 \}$. 
\end{remark}

In the case $k = 2$, the automorphism $\theta$ is trivial, and the $\sigma$-type 
irreducible modules for $K(\mathfrak{sl}_2,2) \cong L(1/2,0)$ are $M^0$ and 
$M^{2,1} \cong L(1/2,1/2)$. 
Thus \eqref{fusion-3_e} corresponds to the fusion product 
$L(1/2,1/2) \boxtimes_{M^0} L(1/2,1/2) = L(1/2,0)$. 

The next theorem follows from Theorem \ref{thm:fusion_M2jj01} and 
Remark \ref{rem:fusion_alg_generator}.

\begin{theorem}\label{thm:Z2_sym}
Let $k \ge 3$ be an integer. Then a map defined by
\[
  (M^{2j,j})^{\ep} \mapsto (-1)^{j+\ep}(M^{2j,j})^{\ep} \quad \text{for} \ 
  0 \le j \le \lfloor k/2 \rfloor \ \text{and} \ \ep \in \{ 0,1 \}
\] 
is compatible with 
\eqref{eq:fusion-1}, \eqref{eq:fusion-2}, \eqref{fusion-3_o}, 
and \eqref{fusion-3_e}, 
and it induces an automorphism of order $2$ of the subalgebra 
of the fusion algebra of $M^{0,+}$ spanned by 
$(M^{2j,j})^{\ep}$ for $0 \le j \le \lfloor k/2 \rfloor$ and $\ep \in \{ 0,1 \}$, 
where $M^{0,+} = K(\mathfrak{sl}_2,k)^{\la \theta \ra}$.
\end{theorem}

Let $V$ be a vertex operator algebra containing 
a vertex operator subalgebra $W \cong K(\mathfrak{sl}_2,k)$. 
Suppose $V^{\la \tau_W \ra} = V$, 
that is, $V$ is of $\sigma$-type as a $W$-module. 
Denote by $V_{W^+}[j,\ep]$ the sum of all irreducible 
$M^{0,+}$-submodules of $V$ isomorphic to $(M^{2j,j})^{\ep}$. 
Then 
\begin{equation}\label{eq:dec_Vtau}
  V = \bigoplus_{j=0}^{\lfloor k/2 \rfloor} \bigoplus_{\ep \in \{ 0,1 \}} 
  V_{W^+}[j,\ep].
\end{equation}

The following theorem is a consequence of Theorem \ref{thm:Z2_sym} 
and \eqref{eq:dec_Vtau}. 

\begin{theorem}\label{thm:sigma-inv}
Let $V$ be a vertex operator algebra containing 
a vertex operator subalgebra $W \cong K(\mathfrak{sl}_2,k)$ 
for an integer $k \ge 3$. 
Assume that $V$ is of $\sigma$-type as a $W$-module. 
Then a linear map $\sigma_W$ on $V$ defined by 
\[
  \sigma_W(v) = (-1)^{j+\ep}v \quad \text{for} \ 
  v \in V_{W^+}[j,\ep],  0 \le j \le \lfloor k/2 \rfloor, \ep \in \{ 0,1 \}
\]
is an automorphism of the vertex operator algebra $V$ of order $2$.  
\end{theorem}

We call $\sigma_W$ the $\sigma$-involution of $V$ associated with $W$. 
We say $W \cong K(\mathfrak{sl}_2,k)$ 
is a $\sigma$-type parafermion vertex operator subalgebra of $V$ 
or of $\sigma$-type in $V$
if $V$ is of $\sigma$-type as a $W$-module.

\section{$\sigma$-involution of $V_{\sqrt{2}A_{k-1}}$}\label{sec:s_VN}

In this section, we study $\sigma$-involutions of 
the lattice vertex operator algebra $V_{\sqrt{2}A_{k-1}}$ for an integer $k \ge 3$. 
Furthermore, we obtain a sufficient condition on a positive definite even lattice $L$ 
containing $\sqrt{2}A_{k-1}$ for which $V_L$ possesses a $\sigma$-involution, 
and show how a $\sigma$-involution is related to an RSSD involution. 

We use the notation in Sections 3, 4, and 5 of \cite{AYY2019}.
Let $L^{(k)} = \Z\alpha_1 + \cdots + \Z\alpha_k$ 
with $\la \alpha_i, \alpha_j \ra = 2 \delta_{i, j}$, and set 
$\gamma_k = \alpha_1 + \cdots + \alpha_k$. Let
\[
  N = \{\al \in L^{(k)} \mid \la \al, \gamma_k \ra = 0 \}.
\]
Then $N \cong \sqrt{2}A_{k-1}$. 
Let $\beta_i = \alpha_i - \alpha_{i+1}$,  
so $\{\beta_1, \ldots, \beta_{k-1}\}$ is a $\Z$-basis of $N$. 

Since $\la \alpha, \alpha \ra \in 4\Z$ for $\alpha \in N$, 
we may take a $\Z$-bilinear map $\vep : N \times N \to \Z_2$ satisfying 
\eqref{eq:cond_vep} for $N$ in place of $L$ to be $\vep = 0$. 
Then the central extension \eqref{eq:ext_L} of $N$ in place of $L$ splits. 
So the twisted group algebra $\C[N]_\vep$ 
is isomorphic to the ordinary group algebra $\C[N]$,  
and $V_N = M(1) \otimes \C[N]$ as a vector space. 
Moreover, we may take a quadratic form $\eta$ satisfying \eqref{eq:def-eta} 
for $N$ to be $\eta = 0$, so that
$\Aut(V_N) = N(V_N) : O(N)$ is a split extension of $N(V_N)$ by $O(N)$. 

The vertex operator algebra $V_N$ contains a subalgebra 
\[
  T \cong L(c_1,0) \otimes \cdots \otimes L(c_{k-1},0) \otimes M^0, 
\]
where $L(c_m,0)$ is a simple Virasoro vertex operator algebra of central charge
\[
  c_m = 1 - \frac{6}{(m+2)(m+3)},
\]
and $M^0 = K(\mathfrak{sl}_2,k)$. 
In fact, $M^0$ is the commutant of 
$L(c_1,0) \otimes \cdots \otimes L(c_{k-1},0)$ in $V_N$. 
The lift $\theta \in \Aut(V_N)$ of the $-1$-isometry of $N$ leaves $M^0$ invariant, 
and its restriction to $M^0$ generates $\Aut (M^0)$. 

The vertex operator algebra $V_N$ decomposes into a direct sum 
of irreducible $T$-modules. 
\begin{equation*}
  V_N = 
  \bigoplus_{\substack{
  0 \le i_s \le s\\
  i_s \equiv 0 \pmod{2}\\
  1 \le s \le k}}
  L(c_1, h^1_{i_1+1, i_2+1}) \otimes \cdots \otimes 
  L(c_{k-1},h^{k-1}_{i_{k-1}+1, i_{k}+1}) \otimes M^{i_{k}, i_{k}/2},
\end{equation*}
where
\begin{equation*}
  h^m_{r,s} = \frac{\big( r(m+3) - s(m+2) \big)^2 - 1}{4(m+2)(m+3)}
\end{equation*}
for $1 \le r \le m+1$ and $1 \le s \le m+2$, and $L(c,h)$ is an irreducible 
highest weight $L(c,0)$-module with the highest weight $h$, 
see \cite{Wang1993} and \cite[Theorem 5.2]{AYY2019}. 
Since $M^{i_{k}, i_{k}/2}$ is a $\sigma$-type irreducible $M^0$-module, 
the next lemma holds. 

\begin{lemma}\label{lem:VN_sigma}
$V_N$ is of $\sigma$-type as an $M^0$-module.
\end{lemma}

By the above lemma, we can consider the $\sigma$-involution $\sigma_{M^0}$ 
of $V_N$ associated with $M^0$ as in Theorem \ref{thm:sigma-inv}. 
The next lemma will be used to relate the $\sigma$-involution $\sigma_{M^0}$ 
and the involution $\theta$. 

\begin{lemma}\label{lem:wt1_VN}
\textup{(1)} 
The weight one subspace $(V_N)_1$ of $V_N$ agrees with the 
weight one subspace $M(1)_1$ of $M(1)$,  
which is $k-1$ dimensional. 

\textup{(2)}
$(V_N)_1$ is spanned by the top levels of 
\begin{equation}\label{eq:ds-1}
  L(c_1,0) \otimes \cdots \otimes L(c_{p-1},0) \otimes L(c_p, h^p_{1,3}) 
  \otimes L(c_{p+1},h^{p+1}_{3,3}) \otimes \cdots 
  \otimes L(c_{k-1}, h^{k-1}_{3,3}) \otimes (M^{2,1})^0
\end{equation}
for $1 \le p \le k-1$. 
\end{lemma}

\begin{proof} 
Since $N$ is a rank $k-1$ lattice with minimal square norm $4$, 
the assertion (1) holds. 
Let $1 \le p \le k-1$, and set $i_1 = \cdots = i_p = 0$ and 
$i_{p+1} = \cdots = i_k = 2$. 
Then $h^m_{1,1} = 0$ for $1 \le m \le p-1$, 
$h^p_{1,3} = (p+1)/(p+3)$, and $h^m_{3,3} = 2/(m+2)(m+3)$ for 
$p+1 \le m \le k-1$. Hence we have
\[ 
  h^p_{1,3} + h^{p+1}_{3,3} + \cdots + h^{k-1}_{3,3} = \frac{k}{k+2}.
\]

The top level of $(M^{2,1})^0$ agrees with the top level of $M^{2,1}$, 
and it is one dimentional with weight $2/(k+2)$. 
Therefore, the top level of \eqref{eq:ds-1} 
is one dimentional with weight one for any $1 \le p \le k-1$. 
Thus the assertion (2)  holds.
\end{proof}

By a similar argument as in the proof of \cite[Lemma 5.4]{LY2014}, 
we obtain the next theorem. 

\begin{theorem}\label{thm:theta-sigma}
Let $k \ge 3$ be an integer. Then 
$\sigma_{M^0} = \theta$ as automorphisms of $V_N$, 
where $N = \sqrt{2}A_{k-1}$ and $\theta$ is the lift of the $-1$-isometry of $N$. 
\end{theorem}

\begin{proof}
We see from Theorem \ref{thm:sigma-inv} and Lemma \ref{lem:wt1_VN} 
that $\sigma_{M^0}$ acts as $-1$ on $(V_N)_1 = M(1)_1$. 
The automorphism $\theta$ also acts as $-1$ on $M(1)_1$. 
Thus $\sigma_{M^0} \theta$ is the identity on $M(1)_1$,  
so $\sigma_{M^0} \theta = \exp(\beta(0))$ 
for some $\beta \in \C \otimes_\Z N$ by \cite[Lemma 2.5]{DN1999}. 
The confomal vector of $M^0$ is described in (4.4) of \cite{DLY2009} as
\[
  \frac{1}{4k(k+2)} \sum_{\substack{1 \le p,q \le k\\p \ne q}} 
  (\alpha_p - \alpha_q)(-1)^2\1 
  + \frac{1}{k+2} \sum_{\substack{1 \le p,q \le k\\p \ne q}} 
 e^{\alpha_p - \alpha_q}. 
\]
Note that $\exp(\beta(0))$ fixes $(\alpha_p - \alpha_q)(-1)^2\1$ and 
multiplies $e^{\alpha_p - \alpha_q}$ by $\exp(\la \beta, \alpha_p - \alpha_q\ra)$. 
Since both $\sigma_{M^0}$ and $\theta$ fix the confomal vector of $M^0$, 
we have that $\exp(\la \beta, \alpha_p - \alpha_q\ra) = 1$ for all 
$1 \le p,q \le k$, $p \ne q$. 
Thus $\exp(\beta(0)) = 1$, and $\sigma_{M^0} = \theta$. 
\end{proof}

We discuss a condition on a rational lattice $L$ such that 
$N \subset L \subset N^*$ for which the $V_N$-module $V_L$ 
is of $\sigma$-type as an $M^0$-module. 
Let
\begin{equation}\label{eq:def_lambda_p}
  \lambda_i = \frac{1}{2k} \gamma_k - \frac{1}{2} \alpha_i,
  \quad 1 \le i \le k.
\end{equation}
Then $2\lambda_i  \equiv  2\lambda_k \pmod{N}$, 
$\lambda_1 + \cdots + \lambda_k = 0$, and 
$\{ \lambda_2, \ldots, \lambda_k \}$ is a $\Z$-basis of $N^*$ 
\cite[Lemma 4.2]{AYY2019}. 
Let
\begin{equation}\label{def:N_j_a}
  N(j,\bsa) = N - \sum_{i = 1}^k a_i \lm_i + 2 j \lm_k
\end{equation}
for $0 \le j < k$ and $\bsa = (a_1,\ldots,a_k) \in \{ 0,1\}^k$. 
Since $2k \lambda_k \in N$, we may consider $j$ to be modulo $k$. 
Any coset of $N$ in $N^*$ is of the form $N(j,\bsa)$ 
for some $j$ and $\bsa$. However, $j$ and $\bsa$ are not uniquely 
determined, see \cite[Lemma 4.3]{AYY2019}. 

The irreducible $V_N$-module $V_{N(j,\bsa)}$ decomposes into 
a direct sum of irreducible $T$-modules. 
\begin{equation*}
  V_{N(j,\bsa)} = 
  \bigoplus_{\substack{
  0 \le i_s \le s\\
  i_s \equiv b_s \pmod{2}\\
  1 \le s \le k}}
  L(c_1, h^1_{i_1+1, i_2+1}) \otimes \cdots \otimes L(c_{k-1},h^{k-1}_{i_{k-1}+1, i_{k}+1})
  \otimes M^{i_{k}, j + (i_{k} - b_{k})/2},
\end{equation*}
where $b_s = \sum_{i=1}^s a_i$ \cite[Theorem 5.2]{AYY2019}. 
We denote by $\wt(\bsa)$ the number of nonzero entries $a_i$ of 
$\bsa = (a_1,\ldots,a_k)$. 
Then $b_k = \wt(\bsa)$,  
and the irreducible $M^0$-module $M^{i_{k}, j + (i_{k} - b_{k})/2}$ 
is of $\sigma$-type 
if and only if $i_k$ is even and $2j = \wt(\bsa)$. 
In fact, $M^{i_{k}, j + (i_{k} - b_{k})/2} = \widetilde{M}^{i_k, b_k-2j}$
in the notation of Section \ref{subsec:aut-tW}. 
Thus $V_{N(j,\bsa)}$ is of $\sigma$-type as an $M^0$-module 
if and only if $2j = \wt(\bsa)$. 

Let $X = \frac{1}{2}N$. 
Then $N^*/X \cong \Z_k$. 
For $N \subset L \subset N^*$,  
we have $2L \subset N$ if and only if $L \subset X$. 
We show that $X$ is the union of $N(j,\bsa)$ for all $0 \le j < k$ and 
$\bsa \in \{ 0,1\}^k$ such that $2j = \wt(\bsa)$. 
Indeed, let $0 \le j < k$ and $\bsa = (a_1, \ldots, a_k) \in \{ 0,1\}^k$ 
be such that $2j = \wt(\bsa)$. 
Take $1 \le i_1 < i_2 < \cdots < i_{2j} \le k$ so that $a_i = 1$ if 
$i \in \{ i_1, \ldots, i_{2j} \}$. 
For $1 \le r < s \le k$, we have 
\begin{align*}
  2\lambda_k - \lambda_r - \lambda_s 
  &= \frac{1}{2}(\alpha_r - \alpha_k) + \frac{1}{2}(\alpha_s - \alpha_k)\\
  &\equiv \frac{1}{2}(\alpha_r - \alpha_s) \pmod{N}
\end{align*}
by \eqref{eq:def_lambda_p} as $\alpha_s - \alpha_k \in N$. 
Since $\alpha_r - \alpha_s = \beta_r + \cdots + \beta_{s-1}$, 
it follows that
\begin{align*}
  2j\lambda_k - \sum_{i=1}^k a_i \lambda_i 
  &= \sum_{p=1}^j (2\lambda_k - \lambda_{i_{2p-1}} - \lambda_{i_{2p}})\\
  &\equiv \sum_{p=1}^j 
  \frac{1}{2} (\beta_{i_{2p-1}} + \beta_{i_{2p-1}+1} + \cdots + \beta_{i_{2p}-1}) \pmod{N}. 
\end{align*}
Hence $N(j,\bsa) \subset X$. 

Next, we need to show that any element of $X$ belongs to $N(j,\bsa)$ 
for some $0 \le j < k$ and $\bsa \in \{ 0,1 \}^k$ such that $2j = \wt(\bsa)$. 
For $\bsa = (a_1,\ldots,a_k)$ with $a_r = a_{r+1} = 1$ and 
$a_i = 0$, $i \ne r, r+1$, we have $\beta_r/2 \in N(1,\bsa)$.  
In this case, $N(1,\bsa)$ satisfies the condition $2j = \wt(\bsa)$ with $j = 1$. 
As for the sum of cosets of $N$ in $N^*$, we have
\begin{equation*}
  N(j, \bsa) + N(j', \bsa') 
  = N(j + j' - (\wt(\bsa) + \wt(\bsa') - \wt(\bsa + \bsa'))/2, \bsa + \bsa'),
\end{equation*}
where $\bsa + \bsa'$ is the sum of $\bsa$ and $\bsa'$ as elements of $(\Z_2)^k$, 
that is, the symmetric difference as subsets of $\{0,1\}^k$, 
see Section 4 of \cite{AYY2019}. 
Let
\[
  j'' = j + j' - (\wt(\bsa) + \wt(\bsa') - \wt(\bsa + \bsa'))/2, 
\]
and let $\bsa'' = \bsa + \bsa'$. 
If $2j = \wt(\bsa)$ and $2j' = \wt(\bsa')$, 
then $2j'' = \wt(\bsa'')$. 
Thus for any $d_1, \ldots, d_{k-1} \in \{ 0,1 \}$, 
$d_1\beta_1/2 + \cdots + d_{k-1}\beta_{k-1}/2$ belongs to some $N(j,\bsa)$ 
such that $2j = \wt(\bsa)$ as desired. 

By the above argument, we obtain the following proposition. 

\begin{proposition}\label{prop:s-type_cond}
Let $L$ be a rational lattice 
such that $N \subset L \subset N^*$, 
where $N = \sqrt{2}A_{k-1}$ with $k \ge 3$ an integer.  
Then the $V_N$-module $V_L$ is a $\sigma$-type $M^0$-module 
if and only if $2L \subset N$.
\end{proposition}

We consider a positive definite even lattice containing $N$ as an RSSD sublattice. 

\begin{theorem}\label{thm:N_RSSD_L}
Let $L$ be a positive definite even lattice containing a sublattice 
$N = \sqrt{2}A_{k-1}$ with $k \ge 3$ an integer. 
Suppose $N$ is RSSD in $L$. 

\textup{(1)} $V_L$ is a $\sigma$-type $M^0$-module.

\textup{(2)} If $L(2) = \varnothing$, then $\varphi(\sigma_{M^0}) = t_N$, 
where $\varphi : \Aut(V_L) \to O(L)$ is as in \eqref{eq:ex_seq_Aut_VL}, 
$\sigma_{M^0}$ is the $\sigma$-involution of $V_L$ associated with $M^0$ 
as in Theorem \ref{thm:sigma-inv}, 
and $t_N$ is the RSSD involution of $L$ associated with $N$. 
\end{theorem}

\begin{proof}
Set $B = \Ann_L(N)$. Then $\rank N + \rank B = \rank L$, and 
$V_N \otimes V_B$ is a vertex operator subalgebra of $V_L$, 
see Remark \ref{rmk:VL_VAVB}. 
We have $L \subset L^* \subset N^* \oplus B^*$ as in \eqref{eq:AB_dec}. 
Let $\alpha \in L$. 
Then $\alpha = a + b$ for some $a \in N^*$ and $b \in B^*$. 
Hence $V_{N + B + \alpha} = V_{N+a} \otimes V_{B+b}$ 
as $V_N \otimes V_B$-modules. 
Since $N$ is RSSD in $L$, we have $2a \in N$ by Lemma \ref{lem:prA}. 
Thus $V_{N+a}$ is a $\sigma$-type $M^0$-module 
by Proposition \ref{prop:s-type_cond}. 
This proves the assertion (1). 

Assume that $L(2) = \varnothing$. 
Theorem \ref{thm:theta-sigma} implies that $\varphi(\sigma_{M^0})$ 
is $-1$ on $N$. 
Moreover, $\varphi(\sigma_{M^0})$ is $1$ on $B$ as $V_N \otimes V_B \subset V_L$. 
Hence the assertion (2) holds.
\end{proof}

\section{Centralizer of $\hn$ in $\Aut(V_{\sqrt{2}A_{k-1}})$}
\label{sec:cent_in_aut_VNnu}

We keep the notation in Section \ref{sec:s_VN}. 
Thus $k \ge 3$ is an integer, 
$\la \alpha_i, \alpha_j \ra = 2 \delta_{i, j}$ for $1 \le i, j \le k$, 
$\beta_i = \alpha_i - \alpha_{i+1}$,  and 
$N = \Z \beta_1 + \cdots + \Z\beta_{k-1} \cong \sqrt{2}A_{k-1}$. 
For convenience, we set $\beta_k = \alpha_k - \alpha_1$. 
We regard the indices of $\alpha$ and $\beta$ as elements of $\Z_k$. 
The vertex operator algebra $V_N = M(1) \otimes \C[N]$ is spanned by
\begin{equation}\label{eq:vector_VN}
  \beta_{i_1}(-n_1) \cdots \beta_{i_r}(-n_r) \otimes e^\alpha
\end{equation}
for $r \ge 0$, $1 \le i_s \le k-1$, $n_s > 0$, and $\alpha \in N$. 
Moreover,  
\begin{equation}\label{eq:aut_VN}
  \Aut(V_N) = N(V_N) : O(N)
\end{equation}
is a split extension of $N(V_N)$ by $O(N)$.

Let $\nu \in O(N)$ be an isometry of $N$ induced by a cyclic permutation
\[
  \nu: \alpha_1 \mapsto \alpha_2 \mapsto \cdots \mapsto \alpha_k 
  \mapsto \alpha_1
\]
of order $k$. 
The isometry $\nu$ is fixed point free on $N$. 
In fact, $\nu$ corresponds to a Coxeter element of the Weyl group 
of the root system of type $A_{k-1}$. 
Let $\hn \in \Aut(V_N)$ be a lift of $\nu$. 
Then $\hn$ transforms the vector of $V_N$ in \eqref{eq:vector_VN} as 
\[
  \hn(\beta_{i_1}(-n_1) \cdots \beta_{i_r}(-n_r) \otimes e^\alpha)
  = \nu(\beta_{i_1})(-n_1) \cdots \nu(\beta_{i_r})(-n_r) \otimes e^{\nu \alpha}.
\]
The automorphism $\hn$ is the identity on $M^0$ 
by the definition of $M^0$ in \cite[Section 4]{DLY2009}, 
so it commutes with the $\sigma$-involution $\sigma_{M^0}$ of $V_N$ 
associated with $M^0$.

When we consider $O(N)$ to be a subgroup of $\Aut(V_N)$, 
we write $\nu$ and $-1$ for the automorphisms $\hn$ and $\theta$ 
of $V_N$, respectively. 
The centralizer $C_{\Aut(V_N)}(\hn)$ of $\hn$ in $\Aut(V_N)$ 
is 
\[
  C_{\Aut(V_N)}(\hn) = C_{N(V_N)}(\hn) : C_{O(N)}(\nu)
\]
by \eqref{eq:aut_VN}. 
Moreover, $O(N) = \la -1 \ra \times \Sym_k$, 
and $C_{O(N)}(\nu) = \la -1 \ra \times \la \nu \ra$, 
where a symmetric group $\Sym_k$ of degree $k$ is 
the Weyl group of the root system of type $A_{k-1}$. 
We also have
\begin{equation*}
  \begin{split}
  C_{N(V_N)}(\hn) 
  &= \{\exp(h(0)) \mid h \in 2\pi\sqrt{-1}((1-\nu)N)^*\}\\
  &\cong \Hom(N/(1-\nu)N, \Z_k \}, 
  \end{split}
\end{equation*}
see Remark \ref{rmk:C_AutVL_hatg}. 

Now, $|N/(1-\nu)N| = k$. 
Indeed, $1 + \nu + \cdots + \nu^{k-1} = 0$ on $N$ 
as $\nu$ is fixed point free on $N$ of order $k$, 
Since $\nu^i \beta_1 = \beta_{1+i}$ for $0 \le i \le k-2$, the minimal polynomial 
of $\nu$ on $N$ is $1 + x + \cdots + x^{k-1}$. 
This implies that $|N/(1-\nu)N| = k$, 
see the proof of \cite[Lemma A.1]{GL2011b}. 

Let $\rho$ be the Weyl vector of the root system of type $A_{k-1}$, that is, 
$\rho$ is the half-sum of positive roots. 
Then $\la \rho, \beta_i \ra = \sqrt{2}$ for $1 \le i \le k-1$, 
and
\begin{equation*}
  \la \rho, (1-\nu)\beta_i \ra/\sqrt{2} = 
  \begin{cases}
  0 & \text{if } 1 \le i \le k-2,\\
  k &  \text{if } i = k-1.
  \end{cases}
\end{equation*}
Thus $\rho/\sqrt{2}k \in ((1-\nu)N)^*$. 
Since $|N/(1-\nu)N| = k$ implies $|((1-\nu)N)^*/N^*| = k$, we see that 
$\rho/\sqrt{2}k + N^*$ generates $((1-\nu)N)^*/N^*$. 

Define $\psi \in N(V_N)$ by 
\begin{equation*}
  \psi = \exp(2\pi\sqrt{-1} \rho(0)/\sqrt{2}k).
\end{equation*}
Then $C_{N(V_N)}(\hn) = \la \psi \ra$ is a cyclic group of order $k$, and
\[
  C_{\Aut(V_N)}(\hn) = \la \psi \ra : (\la \theta \ra \times \la \hn \ra). 
\]
The automorphism $\psi$ acts on the vector in \eqref{eq:vector_VN} as 
$\zeta_k^{\la \rho,\alpha \ra/\sqrt{2}}$, 
where $\zeta_k = \exp(2\pi\sqrt{-1}/k)$. 
Moreover, $\theta$ and $\psi$ generate 
a subgroup isomorphic to a dihedral group $\Dih_{2k}$ of order $2k$ 
as $\theta \psi \theta = \psi^{-1}$.

As for the normalizer of $\la \hn \ra$ in $\Aut(V_N)$, 
we have 
\[
  N_{\Aut(V_N)}(\la \hn \ra) 
  = C_{N(V_N)}(\hn) : N_{O(N)}(\la \nu \ra) 
\]
by \eqref{eq:aut_VN}. 
The normalizer of $\la \nu \ra$ in $\Sym_k$ is a split extension of 
$\la \nu \ra$ by the multiplicative group $\Z_k^\times$ consisting 
of invertible elements in $\Z_k$. 
Therefore, the following theorem holds.

\begin{theorem}\label{thm:aut_VNnu}
Let $k \ge 3$ be an integer, and  
let $\theta$, $\hn$, and $\psi$ be as above. 
Then the centralizer of $\hn$ and the normalizer of $\la \hn \ra$ in 
$\Aut(V_{\sqrt{2}A_{k-1}})$ are as follows. 

\textup{(1)} 
$C_{\Aut(V_{\sqrt{2}A_{k-1}})}(\hn) 
= \la \psi \ra : (\la \theta \ra \times \la \hn \ra)$  
with $\la \psi \ra : \la \theta \ra \cong \Dih_{2k}$. 

\textup{(2)} 
$N_{\Aut(V_{\sqrt{2}A_{k-1}})}(\la \hn \ra) 
= \la \psi \ra : (\la \theta \ra \times (\la \hn \ra : \Z_k^\times))$.
\end{theorem}

In order to describe the action of $\Z_k^\times$ on $\la \psi \ra$, 
we consider the root lattice $A_{k-1}$ inside $\R^k$. 
Let $\vep_1, \ldots, \vep_k$ be the unit vectors in $\R^k$. 
Then $\vep_i - \vep_{i+1}$, $1 \le i \le k-1$, 
form the set of simple roots of $A_{k-1}$, and the Weyl vector $\rho$ is 
\[
  \rho = \frac{1}{2} \sum_{i=1}^k (k+1-2i)\vep_i.
\]

Let $r_i$ be the isometry of $\R^k$ induced by the transposition 
of $\vep_i$ and $\vep_{i+1}$. 
Then the isometry $\nu \in O(N)$ corresponds to $\lambda = r_1 \cdots r_{k-1}$. 
For convenience, we regard the index of $\vep$ as an element of $\Z_k$. 
For $s \in \Z_k^\times$, 
let $\tau_s$ be a permutation on $\{ \vep_1, \ldots, \vep_k \}$ 
defined by $\tau_s(\vep_i) = \vep_{si}$, where $si$ is the product of $s$ and $i$ 
in $\Z_k$. 
Then $\tau_s \lambda \tau_s^{-1} = \lambda^s$ as $\lambda(\vep_i) = \vep_{i+1}$. 
The normalizer of $\la \lambda \ra$ in $\Sym_k$ is 
$N_{\Sym_k}(\la \lambda \ra) = \la \lambda \ra : H$, 
where $H = \{ \tau_s \mid s \in \Z_k^\times \} \cong \Z_k^\times$. 

We have $\la \rho, \vep_i - \vep_{i+1} \ra = 1$ and  
$\la \tau_s \rho, \vep_i - \vep_{i+1} \ra = s^{-1}$. 
Hence $\hat{\tau}_s \psi \hat{\tau}_s^{-1} = \psi^{s^{-1}}$, 
where $\hat{\tau}_s$ is a lift of 
the isometry of $N = \sqrt{2}A_{k-1}$ induced by $\tau_s$. 
In particular, $\hat{\tau}_{k-1} \psi \hat{\tau}_{k-1} = \psi^{-1}$. 
Thus $\hat{\tau}_{k-1}\theta$ centralizes $\psi$. 

By the above argument, we obtain the following corollary.

\begin{corollary}\label{cor:aut_VNnu_over_nu}
$N_{\Aut(V_{\sqrt{2}A_{k-1}})}(\la \hn \ra) / \la \hn \ra
\cong \la \psi \ra : (\la \theta \ra \times \Z_k^\times)$, 
where $\la \psi \ra : \la \theta \ra$ is a dihedral group and 
$\la \psi \ra : \Z_k^\times$ is a Frobenius group.
\end{corollary}

Recall that $\sigma_{M^0} = \theta$ in $\Aut(V_N)$ by 
Theorem \ref{thm:theta-sigma}. 
Since $\psi^{2i}\theta = \psi^i \theta \psi^{-i}$, 
the next proposition holds.

\begin{proposition}\label{prop:other_sigma}
$\psi^{2i} \sigma_{M^0} = \psi^i \sigma_{M^0} \psi^{-i}$ 
is the $\sigma$-involution associated with $\psi^{i}(M^0)$. 
There are $k$ or $k/2$ such $\sigma$-involutions for 
$0 \le i \le k-1$ or $0 \le i \le k/2 - 1$ according as $k$ is odd or even. 
\end{proposition}

\section{Automorphism group of $V_{\sqrt{2}A_{p-1}}^{\la \hn \ra}$}
\label{sec:aut_VNnu}

Let $p$ be an odd prime. 
We keep the notation in Sections \ref{sec:s_VN} and \ref{sec:cent_in_aut_VNnu} 
with $k = p$. 
In this section, we determine the automorphism group 
$\Aut(V_N^{\la \hn \ra})$ of the fixed point subalgebra $V_N^{\la \hn \ra}$ of 
$V_N$ by $\hn$, where $N = \sqrt{2}A_{p-1}$ 
and $\hn$ is a lift of the fixed point free isometry $\nu$ of $N$ of order $p$. 
Since $V_N$ is simple, self-dual, rational, $C_2$-cofinite, and of CFT-type, 
the fixed point subalgebra $V_N^{\la \hn \ra}$ 
is also simple, self-dual, rational, $C_2$-cofinite, and of CFT-type 
by \cite{CM2016} and \cite{Miyamoto2015}. 
Thus every irreducible $V_N^{\la \hn \ra}$-module appears in an irreducible 
$\hn^i$-twisted $V_N$-module for some $0 \le i \le p-1$ by 
\cite[Theorem 3.3]{DRX2017}.

We first discuss irreducible $V_N^{\la \hn \ra}$-modules contained 
in an irreducible $\hn$-twisted $V_N$-module. 
Irreducible twisted modules for a lattice vertex operator algebra 
were constructed explicitly, 
see \cite{BK2004}, \cite{DL1996}, \cite{Lepowsky1985}. 
Following \cite[(4.17)]{BK2004} (see also \cite[Remark 3.1]{DLM2000}), 
set
\[
  \h^{(i,\nu)} = \{ h \in \h \mid \nu h = \zeta_p^{-i} h \}, \quad 0 \le i \le p-1,
\]
where $\h = \C \otimes_\Z N$ and $\zeta_p = \exp(2\pi\sqrt{-1}/p)$. 
Then $\h^{(0,\nu)} = 0$, and $\dim \h^{(i,\nu)} = 1$ for $1 \le i \le p-1$. 
Let $\wh[\nu]$ be the $\nu$-twisted affine Lie algebra as in 
(4.3) - (4.5) of \cite{DL1996}, and let $S[\nu]$ 
be the induced $\wh[\nu]$-module as in (4.9) of \cite{DL1996}. 
Define a $\nu$-invariant alternating 
$\Z$-bilinear map $c^\nu : N \times N \to \Z_{2p}$ by 
\begin{equation*}
  c^\nu(\alpha,\beta) = 2 \sum_{i=1}^{p-1} \la i \nu^i(\alpha),\beta \ra + 2p\Z
\end{equation*}
for $\alpha$, $\beta \in N$, see \cite[Remark 2.2]{DL1996}. 
We denote the radical of $c^\nu$ by $R_N^\nu$. 

Let $\hat{N}_\nu$ be a central extension of $N$ by a cyclic group 
$\la \kappa_{2p} \ra$ of order $2p$ with the commutator map $c^\nu$. 
Then the irreducible $\hn$-twisted $V_N$-module constructed in 
\cite{DL1996}, \cite{Lepowsky1985} is of the form
\begin{equation}\label{eq:twisted_VN}
  V_N^{T,\hn} = S[\nu] \otimes T,
\end{equation}
where $T$ is an irreducible $\hat{N}_\nu$-module 
\cite[(4.25)]{DL1996}, \cite[(7.6)]{Lepowsky1985}. 
There are $|R_N^\nu/(1-\nu)N|$ inequivalent irreducible $\hn$-twisted 
$V_N$-modules of such a form \cite[Proposition 6.2]{Lepowsky1985}. 

It follows from \cite[Lemma 3.2]{ALY2018} that 
$R_N^\nu = N \cap (1-\nu)N^*$. 
Since $2\lambda_i = \gamma_p/p - \alpha_i$ 
by \eqref{eq:def_lambda_p} with $k = p$, 
we have $- \beta_i = (1-\nu)(2\lambda_i) \in (1-\nu)N^*$.  
Thus $R_N^\nu = N$. 
Moreover, $|N/(1-\nu)N| = p$ by \cite[Lemma A.1]{GL2011b},  
so there are $p$ inequivalent irreducible $\hn$-twisted 
$V_N$-modules of the form \eqref{eq:twisted_VN}. 
Among the irreducible $V_N$-modules $V_{N(j,\bsa)}$, 
there are exactly $p$ $\hn$-stable ones, namely, 
$V_{N(j,(0,\ldots,0))}$ for $0 \le j \le p-1$. 
Thus any irreducible $\hn$-twisted $V_N$-module is isomorphic 
to $V_N^{T,\hn}$ for some $T$ by \cite[Theorem 10.2]{DLM2000}.  
The confomal weight of $V_N^{T,\hn}$ is 
\begin{equation*}
  \begin{split}
  h(V_N^{T,\hn}) 
  &= \frac{1}{4p^2} \sum_{i=1}^{p-1} i(p-i) \dim \h^{(i,\nu)}\\
  & = \frac{(p-1)(p+1)}{24p}
  \end{split}
\end{equation*}
by \cite[(6.28)]{DL1996}, which is not an integer. 

Since $\hn$ has prime order, 
the above argument for $\hn$ can be applied to any $\hn^i$, $1 \le i \le p-1$. 
Therefore, we have that the conformal weight of any irreducible $\hn^i$-twisted 
$V_N$-module is not an integer for $1 \le i \le p-1$. 

Next, we discuss irreducible $V_N^{\la \hn \ra}$-modules contained in 
an irreducible untwisted $V_N$-module $V_{N(j,\bsa)}$. 
Note that the conformal weight of any irreducible 
$V_N^{\la \hn \ra}$-module except for $V_N^{\la \hn \ra}$ is positive. 
If $V_{N(j,\bsa)}$ is not $\hn$-stable, then it is an irreducible 
$V_N^{\la \hn \ra}$-module. 
In this case, the quantum dimension of $V_{N(j,\bsa)}$ 
as a $V_N^{\la \hn \ra}$-module is greater than $1$, so 
$V_{N(j,\bsa)}$ is not a simple current $V_N^{\la \hn \ra}$-module 
by \cite[Proposition 4.17]{DJX2013}. 

We see from \cite[Theorem A.1]{AYY2019} that the conformal weight of 
$V_{N(j,(0,\ldots,0))}$ is $j(p-j)/p$ for $0 \le j \le p-1$, which is not an integer 
unless $j = 0$ as $p$ is a prime. 
If $j = 0$, then $N(j,(0,\ldots,0)) = N$. 
Let 
\[
  V_N(i) = \{ v \in V_N \mid \hn v = \zeta_p^{-i}v \}, \quad 0 \le i \le p-1.
\]
Then $V_N(0) = V_N^{\la \hn \ra}$, and the weight one subspace of $V_N(i)$ 
corresponds to $\h^{(i,\nu)}$. 
Moreover, $V_N(i)$, $0 \le i \le p-1$, are simple current 
$V_N^{\la \hn \ra}$-modules by \cite[Theorem 6.3]{DJX2013}. 

By the above argument, we obtain the following lemma.

\begin{lemma}\label{lem:VNi}
Among the irreducible $V_N^{\la \hn \ra}$-modules, only 
$V_N(i)$, $1 \le i \le p-1$, are simple currents with nonzero weight one subspace.
\end{lemma}

Let $g$ be an automorphism of $V_N^{\la \hn \ra}$. 
Then for each $1 \le i \le p-1$, Lemma \ref{lem:VNi} implies that
$V_N(i) \circ g = V_N(j)$ for some $1 \le j \le p-1$. 
Hence $g$ can be extended to an automorphism of the vertex operator algebra 
$V_N$ by \cite[Theorem 2.1]{Shimakura2007} 
as $V_N = \bigoplus_{i=0}^{p-1} V_N(i)$ is a $\Z_p$-graded simple current 
extension of $V_N^{\la \hn \ra}$. 
Moreover, the next theorem holds by \cite[Corollary 2.2]{Shimakura2007} 
and Corollary \ref{cor:aut_VNnu_over_nu}. 

\begin{theorem}\label{thm:Aut_VNnu}
Let $p$ be an odd prime. 
Then 
$\Aut(V_{\sqrt{2}A_{p-1}}^{\la \hn \ra}) 
\cong N_{\Aut(V_{\sqrt{2}A_{p-1}})}(\la \hn \ra)/\la \hn \ra$, 
which has the shape $p : (2 \times (p-1))$ with 
$p:2$ a dihedral group of order $2p$,  
and $p:(p-1)$ a Frobenius group of order $p(p-1)$.
\end{theorem}

\section{Examples}\label{sec:examples}

In this section, 
we discuss $\sigma$-involutions of certain lattice vertex operator algebras.
Those examples illustrate the relationship between $\sigma$-involutions 
and RSSD involutions. 
We also deal with $\sigma$-involutions not related to $V_{\sqrt{2}A_{k-1}}$. 

\subsection{$\sigma$-involutions of $V_{A_{p-1} \otimes R}$}
\label{subsec:lattice_tensor}

Let $p \ge 3$ be an integer, 
and let $R$ be a root lattice of type $A$, $D$, or $E$ of rank $n$. 
We study $\sigma$-involutions of a vertex operator algebra 
$V_{A_{p-1} \otimes R}$ associated with the tensor product 
$A_{p-1} \otimes R$ of $A_{p-1}$ and $R$. 
The tensor product of two lattices $(A, \la \,\cdot\,,\,\cdot\,\ra_A)$ and 
$(B, \la \,\cdot\,,\,\cdot\,\ra_B)$ is by definition the tensor product 
$A \otimes_\Z B$ of $\Z$-modules $A$ and $B$ equipped with 
a symmetric $\Z$-bilinear form 
$\la \,\cdot\,,\,\cdot\,\ra$ defined by
\[
  \la \alpha \otimes \beta, \alpha' \otimes \beta' \ra 
  = \la \alpha, \alpha' \ra_A \cdot \la \beta, \beta' \ra_B 
\]
for $\alpha$, $\alpha' \in A$ and $\beta$, $\beta' \in B$. 
For simplicity of notation, we denote $A \otimes_\Z B$ by $A \otimes B$. 

Let $\vep_1, \ldots, \vep_p$ be the unit vectors in $\R^p$, 
and set $\alpha_i = \vep_i - \vep_{i+1}$ for $1 \le i \le p-1$, 
and set $\alpha_0 = \vep_p - \vep_1$. 
We also take the set $\{ \beta_1, \ldots, \beta_n \}$ of simple roots of $R$. 
Then $\alpha_i \otimes \beta_j$, 
$1 \le i \le p-1$, $1 \le j \le n$, form a $\Z$-basis of $A_{p-1} \otimes R$. 
Note that the symbols $\alpha_i$ and $\beta_j$ here are different from those 
used in Sections \ref{sec:s_VN}, \ref{sec:cent_in_aut_VNnu}, 
and \ref{sec:aut_VNnu}. 

Set $\CA_\beta = A_{p-1} \otimes \Z\beta$ for $\beta \in R$, 
and set $\CA_R = A_{p-1} \otimes R$. 
Then $\CA_\beta \cong \sqrt{2}A_{p-1}$ for $\beta \in R(2)$, 
where $R(2) = \{ \beta \in R \mid \la \beta,\beta \ra = 2\}$. 
The lattice $\CA_R$ is positive definite and even. 
We slightly generalize \cite[Lemma 3.3]{GL2011a} as follows. 

\begin{lemma}\label{lem:min_norm_tensor}
\textup{(1)}
$\la x,x \ra \ge 4$ for $0 \ne x \in \CA_R$.

\textup{(2)} 
$\CA_R(4) = 
\{ \alpha \otimes \beta \mid \alpha \in A_{p-1}(2), \beta \in R(2) \}$.
\end{lemma}

\begin{proof}
We have $\CA_R \subset \bigoplus_{i=1}^p \Z\vep_i \otimes R$,  
and $\la \vep_i \otimes x, \vep_j \otimes y \ra = \delta_{i,j}\la x,y \ra$ 
for $x$, $y \in R$. 
Let $0 \ne x \in \CA_R$. 
Then $x = \sum_{i=1}^p \vep_i \otimes b_i$ for some $b_i \in R$, 
and $\la x,x \ra = \sum_{i=1}^p \la b_i,b_i \ra$. 
Let $i_1,\ldots,i_s$ be the indices $i$ for which $b_i \ne 0$. 
Then $s \ge 2$. 
Since $\la b_i,b_i \ra \ge 2$ for $0 \ne b_i \in R$, the assertion (1) holds.

Suppose $\la x,x \ra = 4$. 
Then $s = 2$ and $\la b_{i_1,} b_{i_1} \ra = \la b_{i_2},b_{i_2} \ra = 2$. 
Since $x \in \CA_R$, we have $x = (\vep_{i_1} - \vep_{i_2}) \otimes b_{i_1}$. 
Thus the assertion (2) holds.
\end{proof}

By a similar argument as in the proof of \cite[Proposition 5.24]{LY2014}, 
we obtain the next lemma.

\begin{lemma}\label{CAbeta_in_CAR}
If $\beta \in R(2)$, then $\CA_\beta$ is RSSD in $\CA_R$.  
Moreover, $t_{\CA_\beta} = 1 \otimes r_\beta$ in $O(\CA_R)$, 
where $t_{\CA_\beta}$ is the RSSD involution associated with $\CA_\beta$,   
and $r_\beta: x \mapsto x - \la x, \beta \ra \beta$ 
is the reflection on $R$ associated with $\beta$.
\end{lemma}

\begin{proof}
Let $\beta \in R(2)$. 
Then we can choose the set $\{ \beta_1, \ldots, \beta_n \}$ of simple roots 
of $R$ containing $\beta$, say, $\beta = \beta_i$. 
Since $R$ is a root lattice of type $A$, $D$, or $E$, 
we have $\la \beta, \beta_j \ra = 0$ or $-1$ for $j \ne i$. 
If $\la \beta, \beta_j \ra = 0$, 
then $\CA_{\beta_j} \subset \Ann_{\CA_R}(\CA_\beta)$. 
If $\la \beta, \beta_j \ra = -1$, then $\la \beta, \beta + 2\beta_j \ra = 0$, 
so $\CA_{2\beta_j} \subset \CA_\beta + \Ann_{\CA_R}(\CA_\beta)$. 
Thus $\CA_\beta$ is RSSD in $\CA_R$. 

Both $t_{\CA_\beta}$ and $1 \otimes r_\beta$ are $-1$ on $\CA_\beta$, 
and $1$ on $\CA_{\beta_j}$ for $\beta_j$ such that $\la \beta, \beta_j \ra = 0$. 
If $\la \beta, \beta_j \ra = -1$, then we consider
\[
  \alpha \otimes \beta_j 
  = \alpha \otimes (- \frac{1}{2}\beta) + \alpha \otimes \frac{1}{2}(\beta + 2\beta_j)
\]
for $\alpha \in A_{p-1}$. 
The first term on the right-hand side of the above equation belongs to 
$\Q \CA_\beta$, and the second term belongs to 
$\Q \Ann_{\CA_R}(\CA_\beta)$, so we have
\begin{equation*}
  \begin{split}
  t_{\CA_\beta} (\alpha \otimes \beta_j) 
  &= - \alpha \otimes (- \frac{1}{2}\beta) + \alpha \otimes \frac{1}{2}(\beta + 2\beta_j)\\
  &= \alpha \otimes (\beta + \beta_j), 
  \end{split}
\end{equation*}
which agrees with $(1 \otimes r_\beta)(\alpha \otimes \beta_j)$. 
Therefore, $t_{\CA_\beta} = 1 \otimes r_\beta$ on $\CA_R$.
\end{proof}

Let $\beta \in R(2)$. 
Then $\CA_\beta \cong \sqrt{2}A_{p-1}$, so the vertex operator algebra 
$V_{\CA_\beta}$ contains a subalgebra $W_\beta \cong K(\mathfrak{sl}_2,p)$ 
which corresponds to $M^0 \subset V_N$ 
in the notation of Section \ref{sec:s_VN}. 
The vertex operator algebra $V_{\CA_R}$ is a $\sigma$-type $W_\beta$-module 
by (1) of Theorem \ref{thm:N_RSSD_L} and Lemma \ref{CAbeta_in_CAR}. 
Let $\sigma_{W_\beta} \in \Aut(V_{\CA_R})$ be the $\sigma$-involution 
associated with $W_\beta$ as in Theorem \ref{thm:sigma-inv}. 
Since $\CA_R(2) = \varnothing$ by Lemma \ref{lem:min_norm_tensor}, 
we have an exact sequence
\begin{equation*}
  1 \longrightarrow N(V_{\CA_R}) \longrightarrow \Aut(V_{\CA_R}) 
  \stackrel{\varphi}{\longrightarrow} O({\CA_R}) \longrightarrow 1 
\end{equation*}
of groups as in \eqref{eq:ex_seq_Aut_VL}. 
Moreover, 
\begin{equation}\label{eq:sigma_and_r}
  \varphi(\sigma_{W_\beta}) = 1 \otimes r_\beta
\end{equation}
by (2) of Theorem \ref{thm:N_RSSD_L} and Lemma \ref{CAbeta_in_CAR}. 

Now, we assume that $p$ is an odd prime. 
Let $\nu \in O(A_{p-1})$ be an isometry of $A_{p-1}$ induced by 
\[
  \nu: \vep_1 \mapsto \vep_2 \mapsto \cdots \mapsto \vep_p 
  \mapsto \vep_1.
\]
Then $\nu(\alpha_i) = \alpha_{i+1}$ for $1 \le i \le p-2$, and 
$\nu(\alpha_{p-1}) = \alpha_0$. 
The isometry $\nu$ is fixed point free on $A_{p-1}$ of order $p$. 
We consider $\nu \otimes 1 \in O(\CA_R)$; 
$\alpha \otimes \beta \mapsto (\nu\alpha) \otimes \beta$, 
which is fixed point free on $\CA_R$ of order $p$. 
For simplicity of notation, we also denote $\nu \otimes 1$ by $\nu$. 
Then the restriction of $\nu \in O(\CA_R)$ to $\CA_\beta$ for $\beta \in R(2)$ 
agrees with $\nu \in O(N)$ in Section \ref{sec:cent_in_aut_VNnu}. 

\begin{lemma}\label{lem:quotient_CAR}
$(1-\nu)\CA_R = ((1-\nu)A_{p-1}) \otimes R$, 
and $\CA_R/(1-\nu)\CA_R \cong p^n$ is elementary abelian of order $p^n$. 
\end{lemma}

\begin{proof}
Since $(1-\nu)(\alpha \otimes \beta) = ((1-\nu)\alpha) \otimes \beta$, 
the first assertion holds. 
The second assertion follows from \cite[Lemma A.1]{GL2011b} 
as $p$ is an odd prime.  
\end{proof}

Let $\hn \in \Aut(V_{\CA_R})$ be a lift of $\nu$. 
In order to deal with the restriction of $\hn$ to $V_{\CA_\gamma}$ 
for $\gamma \in R(2)$, 
we recall the definition of $\hn$ in Section \ref{subsec:aut_hatL}. 
For simplicity, we denote $\CA_R$ by $L$. 
Let $\vep : \olL \times \olL \to \Z_2$ be a bilinear form 
satisfying \eqref{eq:cond_vep}, and let 
$\eta : \olL \to \Z_2$ be a quadratic form satisfying
\begin{equation*}
  \eta(\ola + \olb) + \eta(\ola) + \eta(\olb) 
  = \vep(\ola,\olb) + \vep(\overline{\nu\alpha}, \overline{\nu\beta})
  \quad \text{for} \ \alpha, \beta \in L  
\end{equation*}
as in \eqref{eq:def-eta}, 
where $\olL = L/2L$, $\ola = \alpha + 2L$, and $\olb = \beta + 2L$. 
Then  
$e^\alpha \kappa^a e^\beta \kappa^b 
= e^{\alpha+\beta} \kappa^{\vep(\ola,\olb) + a + b}$ 
in $\hat{L}$ as in \eqref{eq:def-prod}, and 
$\hn \in O(\hat{L})$ is given by 
$\hn(e^\alpha\kappa^a) = e^{\nu \alpha} \kappa^{\eta(\ola) + a}$ 
as in \eqref{def-hatg}. 

Since $\la \alpha,\alpha \ra \in 4\Z$ for $\alpha \in \CA_\gamma$, 
the restriction of $\vep$ to 
$\overline{\CA}_\gamma \times \overline{\CA}_\gamma$ 
is an alternating bilinear form, 
where $\overline{\CA}_\gamma = \CA_\gamma + 2L$. 
Note that $2\CA_\gamma \subset \CA_\gamma \cap 2L$.   
Thus there is a quadratic form $\xi : \overline{\CA}_\gamma \to \Z_2$ such that
\begin{equation*}
  \xi(\ola + \olb) + \xi(\ola) + \xi(\olb) = \vep(\ola,\olb) 
  \quad \text{for} \ \alpha, \beta \in \CA_\gamma.
\end{equation*}

By a similar argument as in the proof of Lemma \ref{lem:div_g_Hom}, 
we see that 
$\Hom(\overline{\CA}_\gamma, \Z_2) \to \Hom(\overline{\CA}_\gamma, \Z_2); 
\mu \mapsto \mu + \mu^\nu$ is a $\Z_2$-linear isomorphism, 
where $\mu^\nu(\ola) = \mu(\overline{\nu \alpha})$. 
Hence there is a unique $\mu \in \Hom(\overline{\CA}_\gamma, \Z_2)$ 
such that $\eta = \xi + \xi^\nu + \mu + \mu^\nu$ 
on $\overline{\CA}_\gamma$, 
where $\xi^\nu(\ola) = \xi(\overline{\nu \alpha})$. 
Let $e'^\alpha = e^\alpha \kappa^{(\xi + \mu)(\ola)}$ for 
$\alpha \in \CA_\gamma$. 
Then $e'^\alpha e'^\beta = e'^{\alpha + \beta}$, and 
$\hn(e'^\alpha) = e'^{\nu \alpha}$. 
Therefore, by changing the section $\alpha \mapsto e^\alpha$ with  
$\alpha \mapsto e'^\alpha$ for $\alpha \in \CA_\gamma$, 
we can regard the restriction of $\hn$ to $V_{\CA_\gamma}$ as 
the automorphism $\hn$ of $V_N$ discussed 
in Section \ref{sec:cent_in_aut_VNnu}. 

Let $\rho$ be the Weyl vector of $A_{p-1}$ 
with respect to the set of simple roots $\{ \alpha_1, \ldots, \alpha_{p-1} \}$. 
Then $\la \rho, \alpha_i \ra = 1$ for $1 \le i \le p-1$, 
so it follows that
\begin{equation*}
\la \rho, (1-\nu)\alpha_i \ra = 
  \begin{cases}
  0 &\text{if} \ 1 \le i \le p-2,\\
  p &\text{if} \ i = p-1.
  \end{cases}
\end{equation*}
Hence $\rho \otimes R \subset p((1-\nu)\CA_R)^*$.

\begin{lemma}\label{lem:ker_chi}
For $\gamma \in R$, we have $\rho \otimes \gamma \in p (\CA_R)^*$ 
if and only if $\gamma \in R \cap pR^*$. 
\end{lemma}

\begin{proof}
If $\gamma \in R \cap pR^*$, 
then $\la \rho \otimes \gamma, \alpha \otimes \beta \ra 
= \la \rho, \alpha \ra \la \gamma, \beta \ra \in p\Z$ 
for $\alpha \in A_{p-1}$ and $\beta \in R$,  
so $\rho \otimes \gamma \in p (\CA_R)^*$. 
Conversely, assume that $\rho \otimes \gamma \in p (\CA_R)^*$. 
Then 
$\la \gamma, \beta \ra 
= \la \rho \otimes \gamma, \alpha_i \otimes \beta \ra \in p\Z$ 
for $\beta \in R$, so $\gamma \in R \cap pR^*$.
\end{proof}

\begin{lemma}\label{lem:R_cap_pRast}
$|(R \cap pR^*)/pR| = p$ if $R = A_n$ with $n+1 \equiv 0 \pmod{p}$ 
or $R = E_6$ with $p = 3$.  
Otherwise, $R \cap pR^* = pR$.
\end{lemma}

\begin{proof}

If $|R^*/R|$ is coprime to $p$, then $R \cap pR^* = pR$. 
Since $p$ is an odd prime, $|R^*/R|$ is divisible by $p$ only if 
$R = A_n$ with $n+1 \equiv 0 \pmod{p}$ or $R = E_6$ with $p = 3$. 
In fact, $R^*/R$ is a cyclic group of order $n+1$ if $R = A_n$, 
and a cyclic group of order $3$ if $R = E_6$.  
Thus the assertion holds.
\end{proof}

For $\beta \in R$, define $\psi_\beta \in N(V_{\CA_R})$ by 
\[
  \psi_\beta  = \exp (2\pi\sqrt{-1}(\rho \otimes \beta)(0)/p). 
\] 
Then $\psi_\beta \in C_{N(V_{\CA_R})}(\hn)$ 
by (3) of Theorem \ref{thm:C_AutVL_hatg}. 
A map $R \to C_{N(V_{\CA_R})}(\hn)$; $\beta \mapsto \psi_\beta$ 
is a homomorphism from an additive group to a multiplicative group. 
Its kernel is $R \cap pR^*$ by Lemma \ref{lem:ker_chi}. 
Thus the image $\{\psi_\beta \mid \beta \in R \}$ 
of the homomorphism is isomorphic to $R/(R \cap pR^*)$. 

Let $\beta \in R(2)$. 
It follows from \eqref{eq:sigma_and_r} that
$\sigma_{W_\beta}  \psi_\beta \sigma_{W_\beta} = \psi_\beta^{-1}$.  
Thus $\sigma_{W_\beta}$ and $\psi_\beta$ generate a dihedral group $\Dih_{2p}$  
of order $2p$. 
Note that 
$\psi_\beta^{2i} \sigma_{W_\beta} = \psi_\beta^i \sigma_{W_\beta} \psi_\beta^{-i}$ 
is the $\sigma$-involution associated with $\psi_\beta^i(W_\beta)$ 
for $0 \le i \le p-1$. 

For the set of simple roots $\{ \beta_1, \ldots, \beta_n \}$ of $R$, 
let $W_l = W_{\beta_l}$, $\sigma_l = \sigma_{W_l}$, and $\psi_l = \psi_{\beta_l}$ 
for $1 \le l \le n$.  
The following theorem is a generalization of \cite[Theorem 5.28]{LY2014}. 

\begin{theorem}\label{thm:sec-7-1}
Let $p$ be an odd prime, and let $R$, $\CA_R$, $\hn$, 
$W_l$, $\sigma_l$, and $\psi_l$ for $1 \le l \le n$ be as above. 
Then the subgroup of $C_{\Aut(V_{\CA_R})}(\hn)$ generated by 
$\sigma$-involutions $\psi_l^i \sigma_l$ 
for $i = 0,1$ and $1 \le l \le n$ 
is isomorphic to 

\textup{(1)} $p^{n-1} : \Weyl(R)$ if $R = A_n$ with $n+1 \equiv 0 \pmod{p}$ 
or $R = E_6$ with $p = 3$, 

\textup{(2)} $p^n : \Weyl(R)$ otherwise, 

\noindent
where $n = \rank R$ and $\Weyl(R)$ is the Weyl group of the root system of $R$.
\end{theorem}

\begin{proof}
We have 
\[
  C_{\Aut(V_{\CA_R})}(\hn) \cong C_{N(V_{\CA_R})}(\hn) : C_{O(\CA_R)}(\nu)
\] 
by Proposition \ref{prop:OL_OhatL} and Theorem \ref{thm:C_AutVL_hatg}. 
Since $\psi_l$, $1 \le l \le n$, generate the subgroup 
$\{\psi_\beta \mid \beta \in R \}$ of $C_{N(V_{\CA_R})}(\hn)$ 
isomorphic to $R/(R \cap pR^*)$, 
and since $\varphi(\psi_l^i \sigma_l) = 1 \otimes r_{\beta_l}$ 
by \eqref{eq:sigma_and_r}, the assertion holds 
by Lemma \ref{lem:R_cap_pRast}. 
\end{proof}

\subsection{$\sigma$-involutions of $V_{L_\CC}$}
\label{subsec:lattice_LCC}

We consider $\sigma$-involutions 
of a lattice vertex operator algebra $V_{L_\CC}$, 
where $L_\CC$ is a positive definite even lattice constructed in \cite{ALY2018} 
by using a certain code $\CC$. 
First, we recall the description of $L_\CC$ in \cite[Section 4]{ALY2018}.  
A lattice of more general form $L_{\CC \times \CD}$ was studied 
in \cite{ALY2018}.   
Since we only deal with the case $\CD = \{0\}$, 
we simply write $L_\CC$ for $L_{\CC \times \{0\}}$. 

Let $p \ge 3$ be an odd integer. 
Let $\alpha_1, \ldots, \alpha_p$, $\beta_1, \ldots, \beta_{p-1}$, 
and $N = \sqrt{2}A_{p-1}$ be as in Section \ref{sec:s_VN} with $k = p$. 
Thus $\la \alpha_i, \alpha_j \ra = 2 \delta_{i,j}$, and 
$\beta_i = \alpha_i - \alpha_{i+1}$, $1 \le i \le p-1$, form a $\Z$-basis of $N$. 
Set
\[
  \beta_u = \frac{1}{2} \sum_{i=1}^{p-1} u_i \beta_i \in N^*
\]
for $u = (u_1,\ldots,u_{p-1}) \in \Z^{p-1}$.  
Then
\[
  \la \beta_u,\beta_v \ra = \frac{1}{2} u A v^t \in \frac{1}{2}\Z
\]
for $u$, $v \in \Z^{p-1}$, 
where $A$ is a $(p-1) \times (p-1)$ matrix with $(i,j)$ entry 
$\la \beta_i, \beta_j \ra/2$, that is, the Cartan matrix of type $A_{p-1}$. 
Note that $\la \beta_u,\beta_u \ra \in \Z$. 
Set 
\[
  L(u) = N + \beta_u.
\]
Since $L(u) = L(v)$ if and only if $u \equiv v \pmod{2\Z^{p-1}}$,  
we may write $L(\olu)$ for $L(u)$, 
where $\olu = (\overline{u_1},\ldots,\overline{u_{p-1}}) \in \Z_2^{p-1}$ with 
$\overline{u_i} = u_i + 2\Z$. 

We have $u A v^t \equiv u' A v'^t \pmod{2\Z}$ 
if $u \equiv u'$ and $v \equiv v' \pmod{2\Z^{p-1}}$,  
so we can define an inner product $\olu \cdot \olv$ on $\Z_2^{p-1}$ by 
\begin{equation}\label{eq:ip_K}
  \olu \cdot \olv = u A v^t + 2\Z \in \Z_2.
\end{equation}
Since the determinant of the matrix $A$ is $\det A = p \notin 2\Z$, 
the inner product \eqref{eq:ip_K} on $\Z_2^{p-1}$ is nondegenerate.
 
For $\alpha$, $\beta \in N^*$ with $\alpha - \beta \in N$, 
we have $\la \alpha, \alpha \ra \equiv \la \beta, \beta \ra \pmod{2\Z}$. 
Define $w(\olu) \in \Z$ by
\begin{equation}\label{eq:wt_K}
  w(\olu) = \min\{ \la x,x \ra \mid x \in L(\olu) \} \in \Z.
\end{equation}
Then $w(\olu) \equiv \la \beta_u, \beta_u \ra \pmod{2\Z}$. 
We also define a map $q: \Z_2^{p-1} \to \Z_2$ by 
\begin{equation}\label{eq:q_K}
  q(\olu) = w(\olu) + 2\Z = \frac{1}{2}u A u^t + 2\Z.
\end{equation}
Then $q(\olu + \olv) + q(\olu) + q(\olv) = \olu \cdot \olv$, 
that is, $q$ is a quadratic form on $\Z_2^{p-1}$ with associated bilinear form 
$\olu \cdot \olv$. 
We denote by $\CK$ the vector space $\Z_2^{p-1}$ over $\Z_2$ equipped with 
the inner product $\olu \cdot \olv$ defined in \eqref{eq:ip_K} and 
the quadratic form $q(\olu)$ defined in \eqref{eq:q_K}. 

Let $d$ be a positive integer. 
For $\bu = (u_1,\ldots,u_d) \in (\Z^{p-1})^d$, set
\[
  \beta(\bu) = (\beta_{u_1},\ldots,\beta_{u_d}) \in (N^*)^d, 
\]
where $(N^*)^d$ is an orthogonal sum of $d$ copies of $N^*$. 
Moreover, set
\[
  L(\bu) = N^d + \beta(\bu). 
\]
We also write $L(\olbu)$ for $L(\bu)$, 
where $\olbu = (\overline{u_1},\ldots,\overline{u_d}) \in \CK^d$. 

We extend the nondegenerate inner product $\olu \cdot \olv$ on $\CK$ 
defined in \eqref{eq:ip_K} to a nondegenerate inner product on $\CK^d$ by 
\begin{equation}\label{eq:ip_Kd}
  \olbu \cdot \olbv = \sum_{i=1}^d \overline{u_i} \cdot \overline{v_i}
\end{equation}
for $\olbu = (\overline{u_1},\ldots,\overline{u_d})$, 
$\olbv = (\overline{v_1},\ldots,\overline{v_d}) \in \CK^d$. 
Then $\olbu \cdot \olbv = 2 \la \beta(\bu), \beta(\bv) \ra + 2\Z$. 
Likewise, we extend the map $w : \CK \to \Z$ defined in \eqref{eq:wt_K} to 
a map $\CK^d \to \Z$ by 
\begin{equation}\label{eq:wt_Kd}
  w(\olbu) = \sum_{i=1}^d w(\overline{u_i}).
\end{equation}
Then $w(\olbu) = \min \{ \la x,x \ra \mid x \in L(\olbu) \}$, 
and $w(\olbu) \equiv \la \beta(\bu),\beta(\bu) \ra \pmod{2\Z}$. 
We also define $q : \CK^d \to \Z_2$ by 
\begin{equation}\label{eq:q_Kd}
  q(\olbu) = w(\olbu) + 2\Z,
\end{equation}
which is a quadratic form on $\CK^d$ with associated bilinear form 
$\olbu \cdot \olbv$. 

For a $\Z$-submodule $\CC$ of $\CK^d$, set
\[
L_\CC = \bigcup_{\olbu \in \CC} L(\olbu). 
\]
Then $L_\CC$ is a sublattice of $(N^*)^d$ containing $N^d$. 
We have the following lemma, 
see Corollary 4.4 and Proposition 4.5 of \cite{ALY2018} in more general form.

\begin{lemma}
\textup{(1)} 
$L_\CC$ is integral if and only if $\CC$ is self-orthogonal 
with respect to the inner product defined in \eqref{eq:ip_Kd}. 

\textup{(2)} 
$L_\CC$ is even if and only if $\CC$ is totally isotropic 
with respect to the quadratic form $q$ defined in \eqref{eq:q_Kd}.
\end{lemma}

As for $L_\CC(2) = \{ \alpha \in L_\CC \mid \la \alpha, \alpha \ra = 2 \}$, 
we have $L_\CC(2) = \varnothing$ if and only if 
$w(\olbu) \ne 2$ for any $\olbu \in \CC$. 
For example, if the number of nonzero entries $\overline{u_i}$ of 
$\olbu = (\overline{u_1},\ldots,\overline{u_d})$ is at least $4$ for any 
nonzero $\olbu \in \CC$, then $L_\CC(2) = \varnothing$. 

Now, assume that $p$ is an odd prime. 
Let $\nu$ be as in Section \ref{sec:cent_in_aut_VNnu} with $k = p$. 
We extend $\nu$ to an isometry of $(N^*)^d$ diagonally, that is, 
$\nu(x_1,\ldots,x_d) = (\nu x_1, \ldots, \nu x_d)$ for $x_i \in N^*$. 
Then $\nu$ induces a fixed point free action on $\CK^d$ 
which preserves the inner product and the maps $w$ and $q$ defined in 
\eqref{eq:ip_Kd}, \eqref{eq:wt_Kd}, and \eqref{eq:q_Kd}, respectively. 

Suppose $\CC$ is a $\nu$-invariant $\Z$-submodule of $\CK^d$ 
such that $w(\olbu) \in 2\Z$ for any $\olbu \in \CC$ and such that 
$w(\olbu) \ge 4$ if $\olbu \ne 0$. 
Then $L_\CC$ is a positive definite even lattice with $L_\CC(2) = \varnothing$. 
Moreover, $L_\CC$ is invariant under $\nu$, 
and the restriction of $\nu$ to $L_\CC$ is a fixed point free isometry 
of $L_\CC$ of order $p$. 
Let $\hn \in \Aut(V_{L_\CC})$ be a lift of $\nu$. 
Then the following lemma holds by Corollary \ref{cor:C_AutVL_hatg}.

\begin{lemma}
$C_{\Aut(V_{L_\CC})}(\hn) \cong p^d :  C_{O(L_\CC)}(\nu)$.
\end{lemma}

We denote by $N_l$ the $l$-th direct summand of $N^d$,    
so $N_l \cong \sqrt{2}A_{p-1}$. 

\begin{lemma}\label{lem:Nl_in_LC}
$N_l$ is RSSD in $L_\CC$ for $1 \le l \le d$.
\end{lemma}

\begin{proof}
Since $2L_\CC \subset N^d \subset N_l + \Ann_{L_\CC}(N_l)$, 
the assertion holds.
\end{proof}

Let $W_l \cong K(\mathfrak{sl}_2,p)$ be a subalgebra 
of the vertex operator algebra $V_{N_l}$
which corresponds to $M^0 \subset V_N$ 
in the notation of Section \ref{sec:s_VN}. 
Then $V_{L_\CC}$ is a $\sigma$-type $W_l$-module 
by (1) of Theorem \ref{thm:N_RSSD_L}  
and Lemma \ref{lem:Nl_in_LC}. 
Let $\sigma_l = \sigma_{W_l} \in \Aut(V_{L_\CC})$ be the $\sigma$-involution 
associated with $W_l$ as in Theorem \ref{thm:sigma-inv}. 
Then $\varphi(\sigma_l) = t_{N_l}$ by (2) of Theorem \ref{thm:N_RSSD_L}, 
where $t_{N_l} \in O(L_\CC)$ is the RSSD involution associated with $N_l$. 

As mentioned in Section \ref{subsec:lattice_tensor}, 
we can regard the restriction of $\hn \in \Aut(V_{L_\CC})$ to $V_{N_l}$ as 
the automorphism $\hn$ of $V_N$ discussed 
in Section \ref{sec:cent_in_aut_VNnu}. 
Recall the vector $\rho$ considered in Section \ref{sec:cent_in_aut_VNnu}. 
Since $\la \rho, \beta_i \ra = \sqrt{2}$ for $1 \le i \le p-1$, 
we have $\la \sqrt{2}\rho, (1-\nu)\beta_i/2 \ra = 0$ or $p$ according as 
$1 \le i \le p-2$ or $i = p-1$. 
Thus
\begin{equation}\label{eq:rho_in_dual_N}
  \sqrt{2}\rho 
  \in p ((1-\nu)(\frac{1}{2}\Z\beta_1 + \cdots + \frac{1}{2}\Z\beta_{p-1}))^*.
\end{equation}

We denote by $\rho_l \in N_l$ the vector corresponding to $\rho \in N$.
Then $\sqrt{2}\rho_l \in p((1-\nu)L_\CC)^*$ by \eqref{eq:rho_in_dual_N}. 
Let
\begin{equation*}
   \psi_l = \exp(2\pi\sqrt{-1}(\sqrt{2}\rho_l)(0)/p)
\end{equation*}
for $1 \le l \le d$. 
Then $\psi_l \in C_{N(V_{L_\CC})}(\hn)$ by Theorem \ref{thm:C_AutVL_hatg}. 
Since $N_l$ is orthogonal to $N_{l'}$ for $l \ne l'$, we have that 
$\psi_l$, $1 \le l \le d$, generate $C_{N(V_{L_\CC})}(\hn) \cong p^d$. 

Since $\sigma_l \psi_l \sigma_l = \psi_l^{-1}$, we see that 
$\sigma_l$ and $\psi_l$ generate a dihedral group $\Dih_{2p}$ of order $2p$, 
and that $\psi_l^{2i} \sigma_l = \psi_l^i \sigma_l \psi_l^{-i}$ 
is the $\sigma$-involution associated with $\psi_l^i(W_l)$ 
for $0 \le i \le p-1$. 
Therefore, the following theorem holds. 

\begin{theorem}\label{thm:sec-7-2}
Let $p$ be an odd prime, and let $\CC$, $L_\CC$, $\hn$, 
$W_l$, $\sigma_l$, and $\psi_l$ for $1 \le l \le d$ be as above. 
Then the subgroup of $C_{\Aut(V_{L_\CC})}(\hn)$ generated by 
$\sigma$-involutions $\psi_l^i \sigma_l$ 
for $i = 0,1$ and $1 \le l \le d$ 
is isomorphic to $(\Dih_{2p})^d$.
\end{theorem}

A possible way to obtain more $\sigma$-involutions is the use of the action of $O(L_\CC)$. 
Let $g \in O(L_\CC)$. 
Then $g(N_l) \cong \sqrt{2}A_{p-1}$ is RSSD in $L_\CC$ by Lemma \ref{lem:Nl_in_LC}. 
So a vertex operator subalgebra $W$ of $V_{g(N_l)}$ corresponding to $M^0 \subset V_N$ 
is a $\sigma$-type parafermion vertex operator subalgebra of $V_{L_\CC}$ 
by Theorem \ref{thm:N_RSSD_L}. 
Hence we can consider the $\sigma$-involution $\sigma_W$ of $V_{L_\CC}$ associated with $W$. 
Moreover, $\varphi(\sigma_W) = t_{g(N_l)} \in O(L_\CC)$ is the RSSD involution 
associated with $g(N_l)$. 
Since $h t_{g(N_l)} h^{-1} = t_{hg(N_l)}$ for $h \in O(L_\CC)$, the group 
$\la t_{g(N_l)} \mid g \in O(L_\CC) \ra$ generated by $t_{g(N_l)}$ for $g \in O(L_\CC)$ 
is a normal subgroup of $O(L_\CC)$. 

Assume that $g \in N_{O(L_\CC)}(\la \nu \ra)$. 
Then $g(N_l)$ is $\nu$-invariant, and $t_{g(N_l)}$ commutes with $\nu$. 
Thus $\la t_{g(N_l)} \mid g \in N_{O(L_\CC)}(\la \nu \ra) \ra$ is a normal subgroup of 
$C_{O(L_\CC)}(\nu)$. 
In this case, we also have $W \subset V_{L_\CC}^{\la \hat{\nu} \ra}$. 
We will discuss such an example in Section \ref{subsec:nonstandard} below.

\subsection{Non-standard $\sigma$-involutions and 
$\mathrm{Aut}(V_{L_\mathcal{C}}^{\la \hat{\nu} \ra})$}\label{subsec:nonstandard}

As shown in Theorem \ref{thm:N_RSSD_L}, 
one can associate a $\sigma$-type parafermion vertex operator subalgebra  
$K(\mathfrak{sl}_2,p) \subset V_L$ for an integer $p \ge 3$ 
with an RSSD sublattice $\sqrt{2}A_{p-1}$ 
of a positive definite even lattice $L$. 
For the orbifold $V_{L_\mathcal{C}}^{\la \hat{\nu} \ra}$ of a lattice vertex operator algebra 
$V_{L_\mathcal{C}}$ discussed in Section \ref{subsec:lattice_LCC}, 
it may also contain some special $\sigma$-type parafermion vertex operator algebras 
which are not obtained from RSSD sublattices isometric to $\sqrt{2}A_{p-1}$. 
In \cite{LY2014}, a special case in which $p=3$ and $L_\CC \cong K_{12}$, 
the Coxeter-Todd lattice of rank $12$, has been studied in detail. 
It turns out that a special $\sigma$-type parafermion vertex operator subalgebra 
$K(\mathfrak{sl}_2,3) \cong W_3(4/5)$ is related to some extra automorphism 
of the orbifold $V_{K_{12}}^{\la \hat{\nu} \ra}$ 
which cannot be extended to an automorphism of $V_{K_{12}}$. 
In this section, we discuss another special case 
in which $p=5$ and $L_{\mathcal{C}}$ is a coinvariant sublattice 
of the Leech lattice $\Lambda$ associated with a $5B$ element  
of $O(\Lambda) = Co_0$.

\subsubsection{Rank $16$ lattice $L_\CC$}

From now on, let $p=5$. 
We use the notation in Section \ref{subsec:lattice_LCC} with $p=5$ and $d=4$. 
Thus $N = \spn_\Z\{ \beta_1, \beta_2, \beta_3, \beta_4\} \cong \sqrt{2}A_4$ 
and $\CK \cong \Z_2^{4}$. 
The fixed point free isometry $\nu$ of order $5$ acts as $\nu \beta_i = \beta_{i+1}$, 
where $i$ is considered to be modulo $5$ and 
$\beta_0 = - (\beta_1 + \beta_2 + \beta_3 + \beta_4)$. 
The isometry $\nu$ also acts on $\CK$ as 
\[
  (i_1,i_2,i_3,i_4) \mapsto (i_4, i_1+i_4, i_2+i_4, i_3+i_4).
\]

Set 
\begin{equation}\label{eq:lmd_p5}
  \lambda=\frac{1}5(\beta_1 +2\beta_2+ 3\beta_3 +4\beta_4), 
\end{equation}
which is $2\lambda_k$ with $k = 5$ in the notation of \eqref{eq:def_lambda_p}. 
Then $\lambda \in N^*$ and $\lambda+N$ generates an order $5$ subgroup of $N^*/N$. 
Note that $\la \lambda, \beta_i \ra = 0$ for $i=1,2,3$, 
$\la \lambda, \beta_4 \ra = 2$, and $\la \lambda, \lambda \ra = 8/5$. 
We also have $\nu^i \lambda = \lambda + \beta_0 + \cdots + \beta_{i-1}$ 
for $1 \le i \le 4$. 
In particular, the coset $\lambda+N$ is fixed by $\nu$. 
We extend $\nu$ to an isometry of $(N^*)^4$ diagonally.  

Now, consider a $\Z$-submodule $\CC$ of $\CK^4$ generated by 
\begin{equation*}
\begin{split}
[(1,0,0,0), (1,0,0,0), (1,0,0,0), (1,0,0,0)],\\
[(0,1,0,0), (0,1,0,0), (0,1,0,0), (0,1,0,0)],\\
[(0,0,1,0), (0,0,1,0), (0,0,1,0), (0,0,1,0)],\\
[(0,0,0,1), (0,0,0,1), (0,0,0,1), (0,0,0,1)],\\
[(1,0,0,0), (0,0,1,1), (1,0,1,1), (0,0,0,0)],\\
[(0,1,0,0), (1,1,1,0), (1,0,1,0), (0,0,0,0)],\\
[(0,0,1,0), (0,1,1,1), (0,1,0,1), (0,0,0,0)],\\
[(0,0,0,1), (1,1,0,0), (1,1,0,1), (0,0,0,0)].
\end{split}
\end{equation*} 

Let $c_1, \ldots, c_8$ be those eight elements of $\CC$  
from the top to the bottom in order. Then $\nu$ acts on $\CC$ as
\begin{gather*}
  c_1 \mapsto c_2 \mapsto c_3 \mapsto c_4 \mapsto c_1+c_2+c_3+c_4 \mapsto c_1,\\
  c_5 \mapsto c_6 \mapsto c_7 \mapsto c_8 \mapsto c_5+c_6+c_7+c_8 \mapsto c_5. 
\end{gather*}
Thus $\CC$ is $\nu$-invariant,  
so the lattice $L_\CC$ is invariant under $\nu$. 
We have $|\CC|=2^8$, and $\CC$ is self-dual with respect to the inner product on 
$\CK^4$ defined as in \eqref{eq:ip_Kd}. 
Hence the dual lattice of $L_\CC$ is 
$(L_\CC)^* = (L_{\CC \times \0})^* = L_{\CC \times \Z_5^4}$ by  
\cite[Proposition 4.3]{ALY2018}, and $(L_\CC)^*/L_\CC \cong 5^4$. 
By a direct calculation, we can verify that $w(c) \in \{ 0, 4, 6, 8 \}$ for $c \in \CC$, 
where $w(c)$ is defined as in \eqref{eq:wt_Kd} for $c \in \CK^4$. 
The number of $c \in \CC$ with given value of $w(c)$ is as follows.
\begin{equation}
\begin{array}{rcccc}
w(c): & 0 \ & 4 & 6 & \ 8\\
\ & 1 \ & 130 & 120 & \ 5
\end{array}
\end{equation}

In particular, $L_\CC(2) = \varnothing$ and $L_\CC(4) \ne \varnothing$. 
Therefore, the next lemma holds.

\begin{lemma}\label{lem:LCC_p5}
$L_\CC$ is a positive definite even lattice of rank $16$ 
with $L_\CC(2) = \varnothing$ and $L_\CC(4) \ne \varnothing$   
such that $(L_\CC)^*/L_\CC \cong 5^4$. 
Moreover, $L_\CC$ is invariant under $\nu$, and
the restriction of $\nu$ to $L_\CC$ is a fixed point free isometry of order $5$.
\end{lemma}

We have $(1-\nu) (L_\CC)^* = L_\CC$ as $(1-\nu)\lambda = -\beta_0$, 
see \cite[Remark 5.2]{ALY2018}, 
and $L_\CC/(1-\nu)L_\CC \cong 5^4$ by \cite[Lemma A.1]{GL2011b}. 
Thus $L_\CC/(1-\nu)L_\CC \cong (L_\CC)^*/L_\CC$ in this case. 

Up to the action of $\nu$ on $\CK^4$ 
and the permutation of the four direct summands $\CK$'s of $\CK^4$, 
there are four types of the $130$ elements $c$ of $\CC$ with $w(c) = 4$, namely, 

\begin{tabular}{rl}
I. & $[(1,0,0,0), (1,0,0,0),(1,0,0,0),(1,0,0,0)]$,\\ 
II. & $[(1,1,0,0), (1,1,0,0), (1,1,0,0), (1,1,0,0)]$,\\
III. & $[(1,0,0,0), (0,0,1,1), (1,0,1,1), (0,0,0,0)]$,\\
IV. &$[(1,1,0,0), (0,1,1,1), (1,1,1,1), (0,1,0,0)]$. 
\end{tabular}

The number of $c \in \CC$ with $w(c) = 4$ of those four types is as follows.

\begin{equation}
\begin{array}{rcccc}
\textrm{type}: & \textrm{I}  \ & \textrm{II} \ & \textrm{III} \ & \textrm{IV}\\
\ & 5 \ & 5 \ & 60 \ & 60
\end{array}
\end{equation}

More precisely, the five elements of either type I or type II belong to a single 
$\la \nu \ra$-orbit, while the $60$ elements of either type III or type IV are divided into 
$12$ $\la \nu \ra$-orbits. 
We denote by $\CK_l$ the $l$-th direct summand of $\CK^4$ for $1 \le l \le 4$, and   
consider the action of an alternating group $\mathrm{Alt}_4$ of degree $4$ on the set 
$\{ \CK_1, \CK_2, \CK_3, \CK_4 \}$ of those direct summands. 
This action induces an action of $\mathrm{Alt}_4$ on $\CC$. 
In fact, the $\mathrm{Alt}_4$-orbit containing 
\[
  [(1,0,0,0), (0,0,1,1), (1,0,1,1), (0,0,0,0)]
\]
has $12$ elements of $\CC$, which form a complete set of representatives of the 
$12$ $\la \nu \ra$-orbits of the type III elements. 
Since the action of $\mathrm{Alt}_4$ commutes with the action of $\nu$, 
the $60$ type III elements comprise an orbit under the action of 
$\la \nu \ra \times \mathrm{Alt}_4$. 
Likewise, the $60$ type IV elements comprise an orbit under the action of 
$\la \nu \ra \times \mathrm{Alt}_4$. 
Note that $\mathrm{Alt}_4$ acts trivially on the type I and the type II elements. 

Let $N_l$ be the $l$-th direct summand of $N^4$. 
Then $N^4$ is an orthogonal sum of $N_l$ for $1 \le l \le 4$. 
We consider the action of $\mathrm{Alt}_4$ on the set $\{ N_1^*, N_2^*, N_3^*, N_4^* \}$.  
Since $\CC$ is invariant under the action of $\mathrm{Alt}_4$ on $\CK^4$, 
this action induces an action of $\mathrm{Alt}_4$ on the lattice $L_\CC$. 
Thus the isometry group $O(L_\CC)$ of the lattice $L_\CC$ contains 
$\la \nu \ra \times \mathrm{Alt}_4$.

\begin{remark}\label{rmk:A41} 
Recall the notation $\beta(\bu) \in (N^*)^d$ for $\bu \in (\Z^{p-1})^d$ 
in Section \ref{subsec:lattice_LCC}. 
Let $c \in \CC$ with $w(c) = 4$. 
We consider $c$ to be an element of $\{0,1\}^{16}$, and define $\beta(c)$ 
as $\beta(\bu)$ for $\bu = c$ with $p = 5$ and $d = 4$. 
Then $\beta(c) \in L_\CC$ and $\la \beta(c),\beta(c) \ra = 4$. 
Denote by $A(\beta(c))$ the sublattice of $L_\CC$ spanned by 
$\nu^i(\beta(c))$ for $0 \le i \le 3$. 
Since $\nu$ is fixed point free of order $5$, 
we have $1 + \nu + \nu^2 + \nu^3 + \nu^4 = 0$ on $L_\CC$. 
Thus $A(\beta(c))$ is the $\nu$-invariant sublattice generated by $\beta(c)$. 
We have $\la \beta(c), \nu(\beta(c)) \ra = -2$, $0$, $-2$, or $-1$ according as 
$c$ is of type I, II, III, or IV, respectively. 
Therefore, $A(\beta(c)) \cong \sqrt{2}A_4$ if $c$ is of type I, II, or III, 
and $A(\beta(c)) \cong A_4(1)$ if $c$ is of type IV by \cite[Lemma D.20]{GL2011a}, 
where $A_4(1)$ is a rank $4$ lattice having the Gram matrix 
\begin{equation*}
  \begin{pmatrix}
  4 &-1 &-1 &-1\\
  -1 &4 &-1 &-1\\
  -1 &-1 &4 &-1\\
  -1 &-1 &-1& 4\\
  \end{pmatrix}. 
\end{equation*}
\end{remark}

\subsubsection{Dihedral $\sqrt{2}E_8$-pairs}

We show that $L_\CC$ can be realized as a sum of two sublattices 
isometric to $\sqrt{2}E_8$ \cite{GL2011a}.
Let $\gamma = \beta_1$ and $\delta = \beta_3+\beta_4$. Then 
$\langle \gamma, \gamma \rangle = \langle \delta, \delta \rangle = 4$ 
and $\langle \gamma, \delta \rangle=0$.  
Let $F = (\Z\gamma \oplus \Z\delta)^4 \subset N^4$. 
Then $F \cong \sqrt{2}A_1^8$.
Define   
\[
  M= \spn_\Z \{ F, \frac{1}{2}(\gamma, \gamma, \gamma, \gamma), 
  \frac{1}{2}(\delta, \delta, \delta, \delta), 
  \frac{1}{2}(\gamma, \delta, \gamma+\delta, 0), 
  \frac{1}{2}(\delta, \gamma+\delta, \gamma, 0) \}, 
\]
and $M'=\nu^2(M)$. 
We can express the generators of $M/F \cong \Z_2^4$ as elements of $\Z_2^8$ 
in the following manner.
\begin{align*}
\frac{1}{2}(\gamma, \gamma, \gamma, \gamma) + F \ 
& \leftrightarrow \ (1,1,1,1,0,0,0,0),\\
\frac{1}{2}(\delta, \delta, \delta, \delta) + F \ 
& \leftrightarrow \ (0,0,0,0,1,1,1,1),\\
\frac{1}{2}(\gamma, \delta, \gamma+\delta, 0) + F \ 
& \leftrightarrow \ (1,0,1,0,0,1,1,0),\\
\frac{1}{2}(\delta, \gamma+\delta, \gamma, 0) + F \ 
& \leftrightarrow \ (0,1,1,0,1,1,0,0). 
\end{align*}

Since $\la \gamma/2, \gamma/2 \ra = \la \delta/2, \delta/2 \ra = 1$, 
the symmetric $\Z$-bilinear form $\la \,\cdot\,,\,\cdot\,\ra$ 
on the lattice $(N^*)^4$ induces the standard inner product on $2F^*/F \cong \Z_2^8$, 
and $M/F$ defines a self-dual $\Z_2$-code with minimal weight $4$ in $\Z_2^8$. 
That is, $M/F$ is isomorphic to the $[8,4,4]$ extended Hamming code.  
Therefore, $M$ is isometric to $\sqrt{2}E_8$. 

Note that $F + \nu^2(F) = N^4$. 
Using the elements $c_1, \ldots, c_8$ of $\CC$, we have 
\begin{align*}
\frac{1}{2}(\gamma, \gamma, \gamma, \gamma) + N^4 
&= c_1, 
&\frac{1}{2}(\delta, \delta, \delta, \delta) + N^4 
&= c_3+c_4,\\
\frac{1}{2}(\gamma, \delta, \gamma+\delta, 0) + N^4 
&= c_5,
&\frac{1}{2}(\delta, \gamma+\delta, \gamma, 0) + N^4 
&=c_7+c_8.
\end{align*}

Since $c_1$, $c_3+c_4$, $c_5$, $c_7+c_8$, 
$\nu^2(c_1)$, $\nu^2(c_3+c_4)$, $\nu^2(c_5)$, and $\nu^2(c_7+c_8)$ 
generate $\CC$, it follows that $L_\CC = M+M'$. 
Thus the next lemma holds.

\begin{lemma}\label{lem:LCC_EE8} 
$L_\CC=M+M'$ with $M\cong M'\cong \sqrt{2}E_8$ and $M \cap M' = 0$. 
\end{lemma}

Uniqueness of such a lattice as $L_\CC$ is known. 
In fact, the following theorem holds \cite[Theorem 4.3]{GL2011b}, 
see also \cite[Corollary 7.17]{GL2011a}. 

\begin{theorem}\label{thm:unique_Q} 
There is, up to isometry, a unique positive definite even lattice $Q$ 
of rank $16$ such that $Q^*/Q \cong 5^4$ and $Q(2) = \varnothing$. 
\end{theorem}

\begin{remark}\label{rmk:Lambda_g} 
The lattice $Q$ described in the above theorem was studied in 
\cite{GL2011a, GL2011b} as a sum of two lattices isometric to $\sqrt{2}E_8$. 
It is isometric to $\mathrm{DIH}_{10}(16)$ in the notation of 
\cite[Section 7]{GL2011a}.  
The lattice $Q$ is also isometric to the coinvariant sublattice 
$\Lambda_g = \Ann_{\Lambda}(\Lambda^g)$  
of the Leech lattice $\Lambda$ associated with an isometry $g$ of class $5B$ 
in $O(\Lambda) = Co_0$, where 
$\Lambda^g = \{ \alpha \in \Lambda \mid g\alpha = \alpha \}$ 
\cite[Theorem 1.6 and Corollary 1.7]{GL2011b}, see also Section 5.3 of \cite{LS2017}.  
\end{remark}

By Lemma \ref{lem:LCC_p5}, we may take $Q = L_\CC$. 
Let
\[
  \lambda_1 = (\lambda,0,0,0), \ \lambda_2 = (0,\lambda,0,0), \ 
  \lambda_3 = (0,0,\lambda,0), \ \lambda_4 = (0,0,0,\lambda),
\]
where  $\lambda$ is defined as in \eqref{eq:lmd_p5}. 
Then $\lambda_i + L_\CC$ for $1 \le i \le 4$ generate the discriminant group 
$\CD(L_\CC) = (L_\CC)^*/L_\CC$ of $L_\CC$. 

For $\alpha \in (L_\CC)^*$, let $\overline{\alpha} = \alpha + L_\CC$. 
Since $\la \alpha + x, \beta + y \ra \in \la \alpha, \beta \ra + \Z$ for 
$\alpha, \beta \in (L_\CC)^*$ and $x, y \in L_\CC$, 
we can define a map $f : \CD(L_\CC) \times \CD(L_\CC) \to \Z_5$ by 
\[
  f(\overline{\alpha}, \overline{\beta}) = 5\la \alpha, \beta \ra + 5\Z 
  \quad \text{for } \alpha, \beta \in (L_\CC)^*, 
\]
which is a nondegenerate symmetric bilinear form of plus type on $\CD(L_\CC)$. 
In fact, $f(\overline{\lambda}_i, \overline{\lambda}_j) = 8 \delta_{i,j}$. 
Let $q : \CD(L_\CC) \to \Z_5$ be a map defined by 
\begin{equation}\label{eq:q_on_DLC}
  q(\overline{\alpha}) = \frac{5}{2} \la \alpha, \alpha \ra + 5\Z
  \quad \text{for } \alpha \in (L_\CC)^*. 
\end{equation}
Then $q$ is a non-singular quadratic form on $\CD(L_\CC)$ 
whose associated bilinear form is $f$. 
We have $q(\overline{\lambda}_i) = 4$, and $q(2\overline{\lambda}_i) = 1$ in $\Z_5$. 
The isometry group of $(\CD(L_\CC), q)$ is the general orthogonal group 
$GO(\CD(L_\CC)) = GO(\CD(L_\CC), q) \cong GO^+_4(5)$. 
Note that $GO^+_4(5)$ is isomorphic to 
$(SL_2(5) \circ SL_2(5)) . 2^2$, an extension of the central product 
$SL_2(5) \circ SL_2(5)$ by $2^2$, 
which has three subgroups of index $2$. 

The isometry group $O(Q)$ of the lattice $Q=L_\CC$ was also studied in \cite{GL2011b}. 
We recall some of the properties of $O(Q)$ as follows, 
see Theorem 1.5 and Corollary B.3 of \cite{GL2011b}. 
Let $D$ be the kernel of the action of $O(Q)$ on $\CD(Q)$. 

\begin{theorem}\label{thm:O_of_Q} 
Let $Q$ be as in Theorem \ref{thm:unique_Q}. 

\textup{(1)} 
There exists an embedding $O(Q) \rightarrow \mathrm{Frob}(20) \times GO^+_4(5)$  
such that the image has index 2 and contains neither direct factor, 
where $\mathrm{Frob}(20)$ is a Frobenius group of order $20$. 
The intersection of the image with $\mathrm{Frob(20)}$ is $D$, 
and the intersection of the image with $GO^+_4(5)$ has the shape 
$(SL_2(5) \circ SL_2(5)) : 2$. 

\textup{(2)} 
The action of $O(Q)$ on $\mathcal{D}(Q)$ induces $GO(\mathcal{D}(Q)) \cong GO^+_4(5)$,  
and the kernel $D$ of the action is a dihedral group of order $10$. 
\end{theorem}

Since $2M^* \subset M$ and $2(M')^* \subset M'$, both $M$ and $M'$ are 
RSSD sublattices of $L_\CC$. 

\begin{lemma}\label{lem:tMtMp} 
\textup{(1)} 
The dihedral group $D$ is generated by $t_M$ and $t_{M'}$, where $t_M$ and $t_{M'}$ 
are the RSSD involutions of $L_\CC$ associated with $M$ and $M'$, respectively. 
Moreover, $\nu = t_M t_{M'}$ as elements of $O(L_\CC)$.  

\textup{(2)}
$\langle \nu \rangle$ is a normal subgroup of $O(L_\CC)$. 

\textup{(3)}
The centralizer $C_{O(L_\CC)}(\nu)$ of $\nu$ in $O(L_\CC)$ is isomorphic to 
$\langle \nu \rangle \times ((SL_2(5) \circ SL_2(5)) : 2)$.
\end{lemma}

\begin{proof}
Since $2M^* \subset M$, it follows from \cite[Lemma A.5]{GL2011b} that 
$t_M$ acts trivially on $\CD(L_\CC)$, see also \cite[Notation 1.3]{GL2011b}. 
Thus $t_M \in D$ by (2) of Theorem \ref{thm:O_of_Q} with $Q = L_\CC$. 
Likewise, we have $t_{M'} \in D$, 
so the dihedral group $D$ is generated by $t_M$ and $t_{M'}$. 

Since $\lambda - \nu\lambda \in N$, the isometry $\nu$ also acts trivially on 
$\CD(L_\CC)$. 
Thus $\nu \in D$, and $\langle \nu \rangle = [D,D]$, the derived subgroup of $D$, 
is a normal subgroup of $O(L_\CC)$. 
In particular, $t_M$ inverts $\nu$. 
Since $M' = \nu^2(M)$ implies $t_{M'} = \nu^2 t_M \nu^{-2}$, 
we have $t_M t_{M'} = t_M \nu^2 t_M \nu^{-2} = \nu$. 
Thus the assertions (1) and (2) hold. 
The assertion (3) follow from the assertion (1) and Theorem \ref{thm:O_of_Q}. 
\end{proof}

Recall that $N_l$ is the $l$-th direct summand of $N^4$ for $1 \le l \le 4$.
Let 
\[
  \CA = \{ g(N_1) \mid g \in O(L_\CC) \}.
\]
Since $\langle \nu \rangle$ is a normal subgroup of $O(L_\CC)$, 
any $A \in \CA$ is $\nu$-invariant. 
We have $N_l \in \CA$ for $1 \le l \le 4$ by the action of $\mathrm{Alt}_4 \subset O(L_\CC)$. 

\begin{lemma}\label{lem:sublattice_A} 
Let $A \in \CA$. 

\textup{(1)}
$A$ is a direct summand of a sublattice of $L_\CC$ 
isometric to an orthogonal sum of four members of $\CA$. 

\textup{(2)}
$A$ is RSSD in $L_\CC$.
\end{lemma}

\begin{proof}
Since $N^4$ is an orthogonal sum of $N_l$ for $1 \le l \le 4$, the assertion (1) holds. 
The assertion (2) follows from Lemma \ref{lem:Nl_in_LC}. 
\end{proof}

\begin{lemma}\label{lem:reflect}
Let $A \in \CA$. 
Then the RSSD involution $t_A \in O(L_\CC)$ associated with $A$ acts on 
$\mathcal{D}(L_\CC)$ as a reflection which maps $\lambda_A + L_\CC$ to its negative, 
where $\lambda_A$ is an element of $A^*$ corresponding to 
$\lambda \in N^*$ defined in \eqref{eq:lmd_p5}  
with $q(\lambda_A + L_\CC) = 4$ a square element of $\Z_5$.
\end{lemma}

\begin{proof}
By Lemma \ref{lem:sublattice_A}, $A$ is a direct summand of an orthogonal sum 
of four members of $\CA$.  
Then $t_A$ acts as $-1$ on $A$ and $1$ on the other three summands 
orthogonal to $A$. 
That means it acts on $\mathcal{D}(L_\CC)$ as a reflection 
associated with $\lambda_A + L_\CC$. 
\end{proof}

\begin{lemma}\label{lem:tA}
The subgroup $\langle t_A \mid A \in \mathcal{A} \rangle$ of $C_{O(L_\CC)}(\nu)$ 
generated by the RSSD involutions $t_A$ associated with $A \in \mathcal{A}$ is 
isomorphic to an index $2$ subgroup of $GO^+_4(5)$ having the shape 
$(SL_2(5) \circ SL_2(5)) : 2$. 
\end{lemma}

\begin{proof}
For $g \in O(L_\CC)$, we have $g t_A g^{-1} = t_{g(A)}$. 
Thus $\langle t_A \mid A \in \mathcal{A} \rangle$ is a normal subgroup of $O(L_\CC)$. 
Since $t_A$ commutes with $\nu$, the assertion follows from Theorem \ref{thm:O_of_Q}, 
Lemmas \ref{lem:tMtMp} and \ref{lem:reflect}. 
\end{proof}

For any $A \in \mathcal{A}$, let $W_A \cong K(\mathfrak{sl}_2,5)$ 
be a vertex operator subalgebra of $V_A$ corresponding to $M^0 \subset V_N$ 
in the notation of Section \ref{sec:s_VN}. 
Then $W_A$ is a $\sigma$-type parafermion vertex operator subalgebra of $V_{L_\CC}$ 
by (1) of Theorem \ref{thm:N_RSSD_L} as $A$ is RSSD in $L_\CC$ 
by Lemma \ref{lem:sublattice_A}. Moreover, $\varphi(\sigma_{W_A}) = t_A$ 
by (2) of Theorem \ref{thm:N_RSSD_L}, where $\varphi : \Aut(V_{L_\CC}) \to O(L_\CC)$ 
is as in \eqref{eq:ex_seq_Aut_VL}, and $\sigma_{W_A}$ is the $\sigma$-involution 
of $V_{L_\CC}$ associated with $W_A$. 

Let $\hat{\nu} \in \Aut(V_{L_\CC})$ be a lift of $\nu$. 
It acts trivially on $W_A$, so $W_A \subset V_{L_\CC}^{\la \hat{\nu} \ra}$. 
Let $\rho_A$ be an element of $(1/\sqrt{2})A$ corresponding to 
$\rho \in (1/\sqrt{2})N$ defined in Section \ref{sec:cent_in_aut_VNnu} with $k = 5$.
Then $\sqrt{2}\rho_A \in 5((1-\nu)L_\CC)^*$ by \eqref{eq:rho_in_dual_N} with $p = 5$. 
Let
\[
  \psi_A = \exp(2\pi\sqrt{-1}(\sqrt{2}\rho_A)(0)/5)
\]
for $A \in \mathcal{A}$.  
Then $\psi_A \in C_{N(V_{L_\CC})}(\hat{\nu})$ by Theorem \ref{thm:C_AutVL_hatg}.
Hence $\psi_A^i(W_A) \subset V_{L_\CC}^{\la \hat{\nu} \ra}$ for $0 \le i \le 4$ 
are also $\sigma$-type parafermion vertex operator subalgebras of $V_{L_\CC}$. 
We consider the $\sigma$-involutions of $V_{L_\CC}^{\la \hat{\nu} \ra}$ 
associated with $\psi_A^i(W_A)$'s. 
Let
\[
  \mathcal{W} = \{ \psi_A^{i}(W_A) \mid A \in \mathcal{A}, i = 0, 1 \}. 
\]

\begin{proposition}\label{prop:7-22}
Let $H = \langle \sigma_W \mid W \in \mathcal{W} \rangle$ be the subgroup 
of $\Aut(V_{L_\CC}^{\la \hat{\nu} \ra})$ generated by the $\sigma$-involutions 
$\sigma_W$ associated with $W \in \mathcal{W}$. Then 
\[
  H \cong 5^4 :  ((SL_2(5) \circ SL_2(5)) : 2) 
  \cong C_{\Aut(V_{L_\CC})}(\hat{\nu})/\langle \hat{\nu} \rangle.
\] 
\end{proposition}

\begin{proof}
Since $\varphi(\sigma_{W_A}) = t_A$ for $A \in \mathcal{A}$, the assertion follows from 
Corollary \ref{cor:C_AutVL_hatg}, Theorem \ref{thm:sec-7-2}, 
(3) of Lemma \ref{lem:tMtMp}, and Lemma \ref{lem:tA}.
\end{proof}

\subsubsection{Ising vectors and $U_{5A}$}

We show that there exist some $\sigma$-type parafermion vertex operator subalgebras 
of $V_{L_\CC}^{\la \hat{\nu} \ra}$ which are not contained in $\mathcal{W}$. 
We first review a few facts about Ising vectors and vertex operator algebras 
generated by two Ising vectors.

A weight $2$ vector $e \in V_2$ of a vertex operator algebra $V$ 
is called an Ising vector if the vertex operator subalgebra $\mathrm{VOA}(e)$  
generated by $e$ is isomorphic to the simple Virasoro vertex operator algebra 
$L(1/2,0)$ of central charge $1/2$. 
Given an Ising vector $e$, one can decompose $V$ as 	
\[
  V = V_e[0] \oplus V_e[1/2] \oplus V_e[1/16], 
\]
where $V_e[h]$ is the sum of all irreducible $\mathrm{VOA}(e)$-submodules of $V$ 
isomorphic to $L(1/2,h)$ for $h=0,1/2,1/16$.
Then the linear map $\tau_e$ on $V$ defined by
\begin{equation*}
  \tau_e=
  \begin{cases}
    1 & \text{ on }\ V_e[0]\oplus V_e[1/2],\\
    -1& \text{ on }\  V_e[1/16] 
  \end{cases}
\end{equation*}
is an automorphism of the vertex operator algebra $V$. 
Note that $\tau_e$ agrees with $\tau_W$ defined in Theorem \ref{thm:tau-autom} 
in the case $k=2$. 

Vertex operator algebras $\mathrm{VOA}(e, f)$ generated by two Ising vectors 
$e$ and $f$ have been studied. In \cite{LYY2005, LYY2007}, 
nine such vertex operator algebras were constructed as subalgebras of $V_{\sqrt{2}E_8}$, 
which are denoted as $U_{nX}$ for 
$nX = 1A$, $2A$, $3A$, $4A$, $5A$, $6A$, $4B$, $2B$, and $3C$.
Uniqueness of $\mathrm{VOA}(e, f)$ with $| \tau_e \tau_f | = 5$ was established in 
\cite{Zheng2020} as a subalgebra of a vertex operator algebra over $\R$ of CFT-type 
with trivial weight $1$ subspace. 
We apply Theorem 3.12 of \cite{Zheng2020} 
to a vertex operator algebra over $\C$ in the following form. 
Note that $| \tau_e \tau_f | = 5$ if and only if $(e, f) = 3/2^9$ by \cite{Sakuma2007}.

\begin{theorem}\label{thm:VOAef}
Let $V$ be a vertex operator algebra of CFT-type with trivial weight $1$ subspace 
$V_1 = 0$. Suppose $V$ possesses a positive definite invariant hermitian form 
$(\,\cdot\,,\,\cdot\,)$ with $(\1,\1) = 1$. 
Let $e, f \in V_2$ be Ising vectors such that the product $\tau_e \tau_f$ 
of the involutions $\tau_e$ and $\tau_f$ associated with $e$ and $f$ has order $5$. 
Then the vertex operator subalgebra
$\mathrm{VOA}(e, f)$ generated by $e$ and $f$ is isomorphic to $U_{5A}$.
\end{theorem}

The properties of $U_{5A}$ were studied in \cite{LYY2005, LYY2007}. 
In particular, we have  
$U_{5A}^{\la \tau_e \tau_f \ra} \cong K(\mathfrak{sl}_2,5) \otimes K(\mathfrak{sl}_2,5)$ 
\cite[(3.62)]{LYY2005}.

Recall from Lemma \ref{lem:LCC_EE8} that $L_\CC = M+M'$ is a sum of two  
sublattices $M$ and $M'$ isometric to $\sqrt{2}E_8$ with $M \cap M' = 0$. 
Since both $M$ and $M'$ are doubly even lattices, 
we can pick a $2$-cocycle $\varepsilon$ on $L_\CC$ such that $\varepsilon$ is trivial on $M$ and $M'$. 
For instance, we may define 
a $\Z$-bilinear map $\varepsilon : L_\CC \times L_\CC \to \Z_2$ by 
\[
  \varepsilon(\alpha, \beta) = \langle \alpha_2, \beta_1\rangle + 2\Z 
\]
for $\alpha = \alpha_1+\alpha_2$, $\beta = \beta_1+\beta_2$ 
with $\alpha_1$, $\beta_1 \in M$ and $\alpha_2$, $\beta_2\in M'$. 
Then $\varepsilon(M,M) = \varepsilon(M',M') = 0$, and 
\[
  \varepsilon(\alpha,\alpha) = \frac{1}{2}\la \alpha,\alpha \ra + 2\Z 
\quad \text{for} \ \alpha \in L_\CC.
\]
Thus $\varepsilon(\alpha, \beta)  + \varepsilon(\beta, \alpha) = \la \alpha, \beta \ra + 2\Z$. 

As in \cite[(3.12)]{LYY2007}, we may define two Ising vectors $e_M$ and $e_{M'}$ 
in $V_{L_\CC}$ as follows.  
\begin{equation*}
\begin{split}
  e_M &= \frac{1}{16} \omega_M +\frac{1}{32} \sum_{\alpha\in M(4)} e^\alpha,\\
  e_{M'} &= \frac{1}{16} \omega_{M'} +\frac{1}{32} \sum_{\alpha\in M'(4)} e^\alpha,
\end{split}
\end{equation*}
where $\omega_M$ and $\omega_{M'}$ are the conformal vectors of $V_M$ and $V_{M'}$, 
respectively. Then $\varphi(\tau_{e_M})= t_M$ and $\varphi(\tau_{e_{M'}})= t_{M'}$
by \cite[Lemma 4.1]{LYY2007}, where $\varphi : \mathrm{Aut}(V_{L_\CC}) \to O(L_\CC)$ 
is a group homomorphism defined as in \eqref{eq:ex_seq_Aut_VL}.
Therefore, $\varphi(\tau_{e_M}\tau_{e_{M'}}) = t_Mt_{M'}$. That is, 
\begin{equation}\label{eq:k5_lift_nu}
  \varphi(\tau_{e_M}\tau_{e_{M'}}) = \nu
\end{equation} 
by Lemma \ref{lem:tMtMp}. 
This implies that $| \tau_{e_M}\tau_{e_{M'}} | = 5$.

Since $V_{L_\CC}$ is a lattice vertex operator algebra, 
it possesses a positive definite invariant hermitian form. 
We consider the fixed point subalgebra $V_{L_\CC}^{\la \theta \ra}$ 
of $V_{L_\CC}$ by the lift $\theta$ of the $-1$-isometry of the lattice $L_\CC$. 
The weight $1$ subspace of $V_{L_\CC}^{\la \theta \ra}$ is trivial as 
$L_\CC(2) = \varnothing$. Note that both $\omega_M$ and $\omega_{M'}$ belong to 
$V_{L_\CC}^{\la \theta \ra}$. 
Thus we can apply Theorem \ref{thm:VOAef} to $V_{L_\CC}^{\la \theta \ra}$, 
$e_M$, and $e_{M'}$ to obtain the following lemma.

\begin{lemma}
The vertex operator subalgebra $\mathrm{VOA}(e_M, e_{M'})$ of $V_{L_\CC}^{\la \theta \ra}$ 
generated by $e_M$ and $e_{M'}$ is isomorphic to $U_{5A}$. Moreover, 
\[
  \mathrm{VOA}(e_M, e_{M'})^{\la \tau_{e_M}\tau_{e_{M'}} \ra}
  \cong K(\mathfrak{sl}_2,5) \otimes K(\mathfrak{sl}_2,5).
\]
\end{lemma}

\begin{remark}
Since $L_\CC = Q \cong \Lambda_g \subset \Lambda$ by Remark \ref{rmk:Lambda_g}, 
we have $\mathrm{VOA}(e_M, e_{M'}) \subset V_{L_\CC}^{\la \theta \ra} 
\subset V_{\Lambda}^{\la \theta \ra}$. 
Thus $\mathrm{VOA}(e_M, e_{M'}) \subset V^\natural$, the Moonshine vertex operator algebra.
\end{remark}

Since $\tau_{e_M}\tau_{e_{M'}} \in \Aut(V_{L_\CC})$ is a lift of $\nu$ by  
\eqref{eq:k5_lift_nu}, we may take $\hat{\nu}$ to be $\hat{\nu} = \tau_{e_M}\tau_{e_{M'}}$. 
Using the notation in Section \ref{subsec:paraf_VOA}, (B.23) of \cite{LYY2005} 
can be written as 
\begin{equation}\label{eq:U5A_dec}
  U_{5A} \cong \bigoplus_{j=0}^4 M^j \otimes M^{2j},
\end{equation}
that is, $U_{5A}$ is a $\Z_5$-graded simple current extension of 
$M^0 \otimes M^0 = K(\mathfrak{sl}_2, 5) \otimes K(\mathfrak{sl}_2, 5)$. 

Denote by $W_1$ and $W_2$ the first and the second tensor factor of 
$\mathrm{VOA}(e_M, e_{M'})^{\la \hat{\nu} \ra}
\cong K(\mathfrak{sl}_2,5) \otimes K(\mathfrak{sl}_2,5)$. 
For $i = 1, 2$, one can associate an automorphism $\tau_{W_i}$ of $V_{L_\CC}$ with 
$W_i$ by Theorem \ref{thm:tau-autom}. Then $\tau_{W_1}$ (resp. $\tau_{W_2}$) 
acts on $M^j \otimes M^{2j}$ as $\zeta_5^{3j}$ (resp. $\zeta_5^j$), where 
$\zeta_5 = \exp(2\pi\sqrt{-1}/5)$. 
Hence $\la \hat{\nu} \ra = \la \tau_{W_1} \ra = \la \tau_{W_2} \ra$ as subgroups of 
$\Aut(\mathrm{VOA}(e_M, e_{M'}))$. 
Then the actions of $\la \hat{\nu} \ra$,  $\la \tau_{W_1} \ra$, and $\la \tau_{W_2} \ra$ 
on the $\mathrm{VOA}(e_M, e_{M'})$-module $V_{L_\CC}$ agree with each other. 
In particular, $V_{L_\CC}^{\la \hat{\nu} \ra}$ is of $\sigma$-type as a $W_i$-module 
for $i = 1,2$. 
That is, $W_i$ is a $\sigma$-type parafermion vertex operator subalgebra of 
$V_{L_\CC}^{\la \hat{\nu} \ra}$; however, it is not of $\sigma$-type in $V_{L_\CC}$. 
In fact, $W_i$ is not of $\sigma$-type even in $\mathrm{VOA}(e_M, e_{M'})$. 
Therefore, the corresponding $\sigma$-involutions $\sigma_{W_1}$ and $\sigma_{W_2}$  
are only defined on $V_{L_\CC}^{\la \hat{\nu} \ra}$, 
and they cannot be extended to automorphisms of $V_{L_\CC}$. 

It is known \cite[Theorem5.11]{Lam2020} that the automorphism group 
$\Aut(V_{L_\CC}^{\la \hat{\nu} \ra})$ of $V_{L_\CC}^{\la \hat{\nu} \ra}$ 
is isomorphic to an index $2$ subgroup of the general orthogonal group $GO^+_6(5)$ 
different from $SO^+_6(5)$. We have the following theorem.

\begin{theorem}
The automorphism group $\Aut(V_{L_\CC}^{\la \hat{\nu} \ra})$ of 
the vertex operator algebra $V_{L_\CC}^{\la \hat{\nu} \ra}$ is generated by 
the $\sigma$-involutions $\sigma_W$ for $W \in \mathcal{W} \cup \{ W_1, W_2 \}$. 
\end{theorem}

\begin{proof}
We slightly modify the proof of \cite[Theorem 5.11]{Lam2020} with $Q = L_\CC$ and 
$\hat{g} = \hat{\nu}$. 
Let $G$ be the subgroup of $\Aut(V_{L_\CC}^{\la \hat{\nu} \ra})$ generated by 
$\sigma_W$ for $W \in \mathcal{W} \cup \{ W_1, W_2 \}$. 
Let $N$ and $C$ be as in the proof of \cite[Theorem 5.11]{Lam2020}. 
Actually, the group $C$ agrees with $H = \la \sigma_W \mid W \in \mathcal{W} \ra$ 
defined in Proposition \ref{prop:7-22}. 
Recall that $\sigma_W$ for $W \in \mathcal{W}$ are automorphisms of $V_{L_\CC}$, 
while $\sigma_{W_1}$ and $\sigma_{W_2}$ are only defined on $V_{L_\CC}^{\la \hat{\nu} \ra}$, 
and they cannot be extended to $V_{L_\CC}$. 
Thus $G$ is strictly larger than $H$, and not contained in $N$. 
Then $G$ contains the derived subgroup $\Omega^+_6(5)$ of $GO^+_6(5)$. 
Since $G$ contains a reflection, 
it is an index $2$ subgroup of $GO^+_6(5)$ different from $SO^+_6(5)$. 
Thus $G = \Aut(V_{L_\CC}^{\la \hat{\nu} \ra})$ as desired.
\end{proof}

\appendix

\section{Representations of $U_{5A}$}

As in Section \ref{subsec:nonstandard}, 
let $U_{5A}$ be the vertex operator algebra $U$ constructed in \cite{LYY2005, LYY2007} 
for the $5A$ case.
The irreducible modules for $U_{5A}$ were classified \cite[Theorem 3.19]{LYY2005}, 
and the fusion rules were determined \cite[Theorem 5.3]{DZ2020}. 
In those papers, $U_{5A}$ is studied as an extension of a tensor product 
$L(1/2, 0) \otimes L(25/28, 0) \otimes L(25/28, 0)$ of three Virasoro vertex operator
algebras. In this appendix, we consider $U_{5A}$ as an extension of 
$K(\mathfrak{sl}_2, 5) \otimes K(\mathfrak{sl}_2, 5)$, 
and review the irreducible modules and the fusion rules for $U_{5A}$. 
Actually, $U_{5A}$ is a $\Z_5$-code vertex operator algebra 
\cite[Section 10]{AYY2019}. 
Thus the irreducible $U_{5A}$-modules are known in a more general context  
\cite[Section 8]{AYY2019}. 
We discuss the fusion product of irreducible $U_{5A}$-modules 
by using the induction functor as well.

Recall from \eqref{eq:U5A_dec} that
\[
  U_{5A} \cong \bigoplus_{j=0}^4 M^j \otimes M^{2j}
\] 
is a $\Z_5$-graded simple current extension of 
$M^0 \otimes M^0 = K(\mathfrak{sl}_2, 5) \otimes K(\mathfrak{sl}_2, 5)$.
The vertex operator algebra structure on $U_{5A}$ 
which extends the $M^0 \otimes M^0$-module 
structure is unique \cite[Proposition 5.3]{DM2004}. 
Since $M^0 = K(\mathfrak{sl}_2, 5)$ is simple, self-dual, rational, $C_2$-cofinite, 
and of CFT-type, 
$U_{5A}$ is also simple, self-dual, rational, $C_2$-cofinite, and of CFT-type by
\cite[Theorem 2.14]{Yamauchi2004}. 
Note that the contragredient module of $M^j \otimes M^{2j}$ is  
$M^{-j} \otimes M^{-2j}$.

For a vertex operator algebra $V$, 
we denote by $\Irr(V)$ (resp. $\SC{V}$) the set of equivalence classes of 
irreducible $V$-modules (resp. simple current $V$-modules). 
We consider a map $b_V : \SC{V} \times \Irr(V) \to \Q/\Z$ defined by
\[
  b_V(A,X) = h(A \boxtimes_V X) - h(A) - h(X) + \Z
\]
for $A \in \SC{V}$ and $X \in \Irr(V)$, where $h(X)$ is the conformal weight of $X$. 
It is known \cite[(8.1)]{AYY2019} that
\[
  b_{M^0}(M^p,M^{i,j}) = \frac{p(i-2j)}{5} + \Z
\]
for $0 \le i \le 5$, $0 \le j \le 4$, and $0 \le p \le 4$. Since
\begin{equation}\label{eq:act_MpMq}
  (M^p \otimes M^q) \boxtimes_{M^0 \otimes M^0} (M^{i_1, j_1} \otimes M^{i_2, j_2}) 
  = (M^p \boxtimes_{M^0} M^{i_1, j_1}) \otimes (M^q \boxtimes_{M^0} M^{i_2, j_2}), 
\end{equation}
it follows that 
\begin{equation*}
  \begin{split}
  b_{M^0 \otimes M^0} (M^p \otimes M^q, M^{i_1, j_1} \otimes M^{i_2, j_2}) 
  &= b_{M^0}(M^p, M^{i_1, j_1}) + b_{M^0}(M^q, M^{i_2, j_2})\\
  &= \frac{1}{5}(p(i_1-2j_1) + q(i_2-2j_2)) + \Z
  \end{split}
\end{equation*}
for $0 \le i_1, i_2 \le 5$, $0 \le j_1, j_2 \le 4$, and $0 \le p, q \le 4$. 

There are $225$ inequivalent irreducible $M^0 \otimes M^0$-modules 
$M^{i_1, j_1} \otimes M^{i_2, j_2}$ for $0 \le j_1 < i_1 \le 5$ and $0 \le j_2 < i_2 \le 5$. 
Among them, we can verify that there are exactly $45$ $M^{i_1, j_1} \otimes M^{i_2, j_2}$'s 
for which 
\[
b_{M^0 \otimes M^0} (M^p \otimes M^{2p}, M^{i_1, j_1} \otimes M^{i_2, j_2}) = 0
\]
for $0 \le p \le 4$. 
We denote by $\Irr^0(M^0 \otimes M^0)$ 
the set of those $45$ irreducible $M^0 \otimes M^0$-modules. 

Let $U^0 = U_{5A} \cong \bigoplus_{j=0}^4 M^j \otimes M^{2j}$. We consider
\[
  U^0 \boxtimes_{M^0 \otimes M^0} X 
  = \bigoplus_{j=0}^4 (M^j \otimes M^{2j}) \boxtimes_{M^0 \otimes M^0} X
\]
for $X \in \Irr^0(M^0 \otimes M^0)$. 
For simplicity of notation, we write $[i_1, j_1; i_2, j_2]$ for $M^{i_1, j_1} \otimes M^{i_2, j_2}$. 
We also write $U^0 \boxtimes X$ for $U^0 \boxtimes_{M^0 \otimes M^0} X$. 
Let
\begin{equation*}
\begin{split}
& U^1 = U^0 \boxtimes [5, 0; 4, 2], \qquad U^2 = U^0 \boxtimes [5, 0; 2, 1],\\
& U^3 = U^0 \boxtimes [2, 0; 5, 3], \qquad U^4 = U^0 \boxtimes [2, 0; 4, 0],\\
& U^5 = U^0 \boxtimes [2, 0; 3, 2], \qquad U^6 = U^0 \boxtimes [1, 0; 5, 4],\\
& U^7 = U^0 \boxtimes [1, 0; 4, 1], \qquad U^8 = U^0 \boxtimes [1, 0; 2, 0].
\end{split}
\end{equation*}

Then by \eqref{eq:isom_Mij}, \eqref{eq:paraf_sc}, and \eqref{eq:act_MpMq}, we have 
\begin{equation}\label{eq:U0-U8_dec}
\begin{split}
&U^0 = [5, 0; 5, 0] + [5, 1; 5, 2] + [5, 2; 5, 4] + [5, 3; 5, 1] + [5, 4; 5, 3],\\
&U^1 = [5, 0; 4, 2] + [5, 1; 1, 0] + [5, 2; 4, 1] + [5, 3; 4, 3] + [5, 4; 4, 0],\\
&U^2 = [5, 0; 2, 1] + [5, 1; 3, 1] + [5, 2; 2, 0] + [5, 3; 3, 0] + [5, 4; 3, 2],\\
&U^3 = [2, 0; 5, 3] + [2, 1; 5, 0] + [3, 0; 5, 2] + [3, 1; 5, 4] + [3, 2; 5, 1],\\
&U^4 = [2, 0; 4, 0] + [2, 1; 4, 2] + [3, 0; 1, 0] + [3, 1; 4, 1] + [3, 2; 4, 3],\\
&U^5 = [2, 0; 3, 2] + [2, 1; 2, 1] + [3, 0; 3, 1] + [3, 1; 2, 0] + [3, 2; 3, 0],\\
&U^6 = [1, 0; 5, 4] + [4, 0; 5, 1] + [4, 1; 5, 3] + [4, 2; 5, 0] + [4, 3; 5, 2],\\
&U^7 = [1, 0; 4, 1] + [4, 0; 4, 3] + [4, 1; 4, 0] + [4, 2; 4, 2] + [4, 3; 1, 0],\\
&U^8 = [1, 0; 2, 0] + [4, 0; 3, 0] + [4, 1; 3, 2] + [4, 2; 2, 1] + [4, 3; 3, 1]
\end{split}
\end{equation}
as $M^0 \otimes M^0$-modules. 
The $45$ irreducible $M^0 \otimes M^0$-modules which appear on the right-hand side of 
\eqref{eq:U0-U8_dec} are exactly the $45$ members of $\Irr^0(M^0 \otimes M^0)$. 

We have the following theorem by Theorems 2.14 and 3.2 of \cite{Yamauchi2004}, 
see also Theorem 2.2 and Proposition 2.3 of \cite{AYY2019}.

\begin{theorem}
There are exactly nine inequivalent irreducible $U_{5A}$-modules, which are 
$U^i$, $0 \le i \le 8$.
\end{theorem}

The top level of $M^{i,j}$ is one dimensional, and its weight $h(M^{i,j})$ is given by \eqref{conf_wt_Mij}. 
Since
\[
  h(M^{i_1, j_1} \otimes M^{i_2, j_2}) = h(M^{i_1, j_1}) + h(M^{i_2, j_2}),
\]
we can calculate the conformal weight of each irreducible direct summand on the 
right-hand side of \eqref{eq:U0-U8_dec}. 
Then we see that the weight and the dimension of the top level of 
$U^i$, $0 \le i \le 8$, are as follows.
\begin{center}
\begin{tabular}{c|ccccccccc}
\ & $U^0$ & $U^1$ & $U^2$ & $U^3$ & $U^4$ & $U^5$ & $U^6$ & $U^7$ & $U^8$\\
\hline
\textup{weight} & $0$ & $6/7$ & $2/7$ & $2/7$ & $1/7$ & $4/7$ & $6/7$ & $5/7$ & $1/7$\\
\textup{dimension} & $1$ & $3$ & $1$ & $1$ & $2$ & $5$ & $3$ & $4$ & $2$
\end{tabular}
\end{center}

Next, we discuss the fusion product $U^i \boxtimes_{U^0} U^j$. 
Let $\CC(M^0 \otimes M^0)$ be the category of $M^0 \otimes M^0$-modules. 
We consider two full subcategories of $\CC(M^0 \otimes M^0)$, namely, 
$\mathrm{Rep}^0 U^0$ and $\CC^0(M^0 \otimes M^0)$, 
where $\mathrm{Rep}^0 U^0$ is the braided tensor category of $U^0$-modules, and 
$\CC^0(M^0 \otimes M^0)$ is the $\C$-linear additive braided monoidal category 
with simple objects being the members of $\Irr^0(M^0 \otimes M^0)$. 
By Theorem 2.67 of \cite{CKM2017}, a functor defined by 
\[
  F : \CC^0(M^0 \otimes M^0) \to \mathrm{Rep}^0 U^0; 
  \quad X \mapsto U^0 \boxtimes_{M^0 \otimes M^0} X
\]
is a braided tensor functor. 
Hence 
\begin{equation}\label{eq:act_F}
  (U^0 \boxtimes_{M^0 \otimes M^0} X) \boxtimes_{U^0}  (U^0 \boxtimes_{M^0 \otimes M^0} Y) 
  = U^0 \boxtimes_{M^0 \otimes M^0} (X \boxtimes_{M^0 \otimes M^0} Y)
\end{equation}
for $X, Y \in \Irr^0(M^0 \otimes M^0)$. 
Since 
\[
(M^{i_1, j_1} \otimes M^{i_2, j_2}) \boxtimes_{M^0 \otimes M^0} (M^{i'_1, j'_1} \otimes M^{i'_2, j'_2})
= (M^{i_1, j_1} \boxtimes_{M^0} M^{i'_1, j'_1}) \otimes (M^{i_2, j_2} \boxtimes_{M^0} M^{i'_2, j'_2}),  
\]
we can calculate the fusion product  
$(U^0 \boxtimes_{M^0 \otimes M^0} X) \boxtimes_{U^0} (U^0 \boxtimes_{M^0 \otimes M^0} Y)$ 
by \eqref{eq:paraf_fusion} and \eqref{eq:act_F}. 
In fact, we obtain the following theorem.

\begin{theorem}
The fusion product of irreducible $U_{5A}$-modules is as follows, where 
$i \boxtimes j = k_1 + \cdots + k_r$ implies $U^i \boxtimes_{U^0} U^j = U^{k_1} + \cdots + U^{k_r}$.

\quad $0 \boxtimes i = i$ for $0 \le i \le 8$,

\quad $1 \boxtimes 1 = 0+2$, \quad $1 \boxtimes 2 = 1+2$, \quad $1 \boxtimes 3 = 4$, 
\quad $1 \boxtimes 4 = 3+5$, \quad $1 \boxtimes 5 = 4+5$, 

\quad $1 \boxtimes 6 = 7$, \quad $1 \boxtimes 7 = 6+8$, \quad $1 \boxtimes 8 = 7+8$,

\quad $2 \boxtimes 2 = 0+1+2$, \quad $2 \boxtimes 3 = 5$, 
\quad $2 \boxtimes 4 = 4+5$, \quad $2 \boxtimes 5 = 3+4+5$,
 
\quad $2 \boxtimes 6 = 8$, \quad $2 \boxtimes 7 = 7+8$, \quad $2 \boxtimes 8 = 6+7+8$, 

\quad $3 \boxtimes 3 = 0+3+6$, \quad $3 \boxtimes 4 = 1+4+7$, 
\quad $3 \boxtimes 5 = 2+5+8$, \quad $3 \boxtimes 6 = 3+6$,

\quad $3 \boxtimes 7 = 4+7$, \quad $3 \boxtimes 8 = 5+8$, 

\quad $4 \boxtimes 4 = 0+2+3+5+6+8$, \quad $4 \boxtimes 5 = 1+2+4+5+7+8$,
\quad $4 \boxtimes 6 = 4+7$, 

\quad $4 \boxtimes 7 = 3+5+6+8$, 
\quad $4 \boxtimes 8 = 4+5+7+8$, 

\quad $5 \boxtimes 5 = 0+1+2+3+4+5+6+7+8$, \quad $5 \boxtimes 6 = 5+8$, 
\quad $5 \boxtimes 7 = 4+5+7+8$, 

\quad $5 \boxtimes 8 = 3+4+5+6+7+8$, 

\quad $6 \boxtimes 6 = 0+3$, \quad $6 \boxtimes 7 = 1+4$, \quad $6 \boxtimes 8 = 2+5$, 

\quad $7 \boxtimes 7 = 0+2+3+5$, \quad $7 \boxtimes 8 = 1+2+4+5$, 

\quad $8 \boxtimes 8 = 0+1+2+3+4+5$.
\end{theorem}

\end{document}